\documentclass[12pt]{amsart}
\usepackage{amsmath,amsfonts}
\usepackage{amssymb}
\usepackage{amscd}
\usepackage{amsthm}
\usepackage{yhmath}

\usepackage{epsfig,subfigure}
\usepackage{todonotes}

\usepackage[all]{xy}
\usepackage{color} 
\usepackage{soul}

\numberwithin{equation}{section}

\newtheorem{theorem}[equation]{Theorem}

\newtheorem{proposition}[equation]{Proposition}
\newtheorem{prop}[equation]{Proposition}

\newtheorem{lem}[equation]{Lemma}
\newtheorem{lemma}[equation]{Lemma}

\newtheorem{corollary}[equation]{Corollary}

\theoremstyle{remark}
\newtheorem{remark}[equation]{Remark}

\theoremstyle{definition}
\newtheorem{definition}[equation]{Definition}

\newtheorem{example}[equation]{Example}

\def\XXint#1#2#3{{\setbox0=\hbox{$#1{#2#3}{\int}$}
	\vcenter{\hbox{$#2#3$}}\kern-.5\wd0}}

\newcommand{\Abn}{\operatorname{Abn}}
\newcommand{\Ad}{\operatorname{Ad}}
\newcommand{\ad}{\operatorname{ad}}

\newcommand{\R}{\mathbb R}

\newcommand{\g}{\mathfrak{g}}
\newcommand{\h}{\mathfrak{h}}

\renewcommand{\ker}{\operatorname{Ker}}

\newcommand{\Lie}{\operatorname{Lie}}

\renewcommand{\span}{\operatorname{span}}

\newcommand{\End}{\operatorname{End}}
\newcommand{\dd}{\operatorname{d}}

\newcommand{\acts}{\curvearrowright}

\def\eps{\epsilon}

\begin{document}

\title{Sard Property for the endpoint map on some 
Carnot
%polarized
 groups}

\author[Le Donne]{Enrico Le Donne}
\address[Le Donne]{
Department of Mathematics and Statistics, University of Jyv\"askyl\"a, 40014 Jyv\"askyl\"a, Finland}
\email{enrico.ledonne@jyu.fi}

\author[Montgomery]{Richard Montgomery}
\address[Montgomery]{Mathematics Department, University of California, 
4111 McHenry
Santa Cruz, CA 95064, USA}
\email{rmont@ucsc.edu}

\author[Ottazzi]{Alessandro Ottazzi}
\address[Ottazzi]{Universit\`a di Trento, Trento 38123 Italy
\&
University of New South Wales, NSW 2052 Australia.}
\email{alessandro.ottazzi@gmail.com}

\author[Pansu]{Pierre Pansu}
\address[Pansu]{
Universit\'e Paris-Sud
B\^atiment 425, 91405 Orsay, France}
\email{Pierre.Pansu@math.u-psud.fr}

\author[Vittone]{Davide Vittone}
\address[Vittone]{Universit\`a di Padova, Dipartimento di Matematica Pura ed
Applicata, via Trieste 63, 35121 Padova, Italy
\&
Universit\"at Z\"urich, Institut f\"ur Mathematik,
Winterthurerstrasse 190, 8057 Z\"urich, Switzerland}
\email{vittone@math.unipd.it}


\renewcommand{\subjclassname}{%
 \textup{2010} Mathematics Subject Classification}

\date{\today}

\begin{abstract} In Carnot-Carath\'eodory or sub-Riemannian geometry, one of the major open problems
is whether the conclusions of Sard's theorem 
holds for the endpoint map, a canonical map from an infinite-dimensional  path space to the
underlying finite-dimensional manifold.  
The set of critical values for the endpoint map is also known as 
abnormal set, being 
 the set of endpoints of
abnormal extremals leaving the base point. 
We prove that a strong version of Sard's property holds 
for all step-$2$ Carnot groups  and several other classes of   Lie groups endowed with left-invariant distributions. 
Namely, we prove that the abnormal set
%of the endpoint map 
%leaving the identity,  form 
lies in a proper analytic subvariety.
% in these  cases.
In doing so we examine several  characterizations of the abnormal set in the case of Lie groups.
%In doing so we provide several new characterizations of the abnormal set in the case of Lie groups,
\end{abstract}

%to be sent to\\
%Annales de l'Institut Henri Poincare C: Analyse Non Lineaire 

\maketitle

\tableofcontents
\section{Introduction}   					
%\footnote{We might consider removing the word "Carnot" from the title, since we also discuss other kind of groups, but probably it is better to leave it.
%Should we say "Sard property" or "the Sard property" or "Sard's property" or...?}
 
Let $G$ be a connected Lie group with 
 Lie algebra $\g$. %, viewed as the tangent space of $G$ at the identity element $e$. 
Let $V\subseteq \g$ be a subspace.
Following Gromov \cite[Sec.~0.1]{Gromov1}, we shall call the pair $(G, V)$ a {\em polarized group}.
Carnot groups are examples of polarized groups where   $V$ is the first layer of their stratification.
To any  polarized group $(G, V)$ one associates the
endpoint map:
\begin{eqnarray*}
\End: L^2([0,1],V)& \rightarrow &G\\
u\qquad&\mapsto& \gamma_u(1),
 \end{eqnarray*}
where
$ \gamma_u$ is the curve on $G$ leaving from the origin $e \in G$ with derivative
$ (\dd L_{\gamma(t)} )_e  u (t)$.

The abnormal set of $(G,V)$ 
is the subset $\Abn(e) \subset G$ of all
singular values of the endpoint map. Equivalently, $\Abn(e)$ is the union of all {\em abnormal curves} passing through the origin (see Section~\ref{sec:abnormal}).
If 
the abnormal set
has measure $0$,
then $(G,V)$ is said to satisfy the 
{\em Sard Property}.
Proving the Sard Property in the general context of polarized manifolds is one of the major open problems in sub-Riemannian geometry, see the questions in \cite[Sec.~10.2]{Montgomery} and Problem III in \cite{Agrachev_problems}.
In this paper, we will focus on
the following stronger versions of Sard's property in the context of  groups.

\begin{definition}[Algebraic and Analytic Sard Property]
We say that a polarized group $(G,V)$ satisfies the 
{\em Algebraic}  (respectively, {\em Analytic}) {\em Sard Property}
if
its abnormal set $\Abn(e)$ is contained in a proper real algebraic (respectively, analytic) subvariety of $G$.
\end{definition}
%Of course the algebraic  property makes sense  when the Lie group $G$ admits an algebraic structure, e.g., when $G$ is nilpotent.
%comment: best deleted -- brings to mind `` what is an ``algebraic group''? 
Our main results are summarized by: 
%The methods for proving Sard's property that we discuss in this paper give the following results:
\begin{theorem}\label{main_thm1}
The following Carnot groups satisfy the Algebraic Sard Property:
\begin{enumerate}
\item
Carnot groups of step $2$;
\item
The free-nilpotent group of rank $3$ and step $3$;
\item
The free-nilpotent group of rank $2$ and step $4$;
\item
%The
%Carnot group give by t
The
 nilpotent part of the Iwasawa decomposition of any semisimple Lie group
equipped with the distribution 
defined by the sum of the simple root spaces.
\end{enumerate}
The following polarized groups satisfy the Analytic Sard Property:
\begin{enumerate}
\item[(5)]
Split semisimple Lie groups 
equipped with the distribution given by the subspace of the Cartan decomposition with negative eigenvalue.
\item[(6)]
%Bruhat 
Split semisimple Lie groups 
equipped with the distribution defined by the sum of the nonzero root spaces.
\end{enumerate}
\end{theorem}

Earlier work \cite{Montgomery_Singular_extremals_on_Lie_groups} allows us
%\footnote{If we can remove the word 'split' in (6), then (7) becomes a particular case. I (Ale) didn't find a classification free method to remove the split hypothesis. However, soon I %will be able to ask Michael, who is an expert in the subject and may suggest something} to add a 7th item to the list of Theorem~\ref{main_thm1}:  

\begin{enumerate}
\item[(7)] {\em compact semisimple Lie groups equipped with the distribution defined by the sum of the nonzero root spaces,
(i.e., the     orthogonal  to the maximal torus relative to a bi-invariant metric).}
 \end{enumerate}
%\begin{enumerate}
%\item[(7)] {\em compact semisimple Lie groups equipped with the distribution defined by the sum of the nonzero root spaces.}
%\footnote{This following sentence got removed:
%\\
%distribution that is  orthogonal  to the maximal torus with respect to a bi-invariant metric}
% \end{enumerate}

Case (1) will be proved reducing the problem to the case of a  smooth map between finite-dimensional manifolds  and  applying the classical Sard Theorem to this map.
The proof will crucially use the fact that in a Carnot group of step $2$ each abnormal curve  is   contained in a proper subgroup.
 This latter property may fail for step 3,  see Section~\ref{not lying in any subgroup}. 
 However, a similar strategy together with the notion of {\em abnormal varieties}, see \eqref{abnormal_variety}, might yield a  proof of  Sard Property for general Carnot groups. 
 
The proof of  
cases (2)-(6) is   based on the observation that,
if $\mathcal X $ is a family of contact vector fields (meaning infinitesimal symmetries of the distribution)  vanishing at the identity, then
for any horizontal curve $\gamma$ leaving from the origin with control $u$  we have
\begin{equation*}
(R_{\gamma(1)})_*V+(L_{\gamma(1)})_* V + \mathcal X(\gamma(1))\subset {\rm Im}(\dd \End_u)\subset T_{\gamma(1)}G.
\end{equation*}
Therefore if $g\in G$ is such that
\begin{equation}\label{condition}  (R_g)_*V+(L_g)_* V + \mathcal X(g)=T_g G,
\end{equation}
then $g$ is not a singular value of the endpoint map. %\notin \Abn(e)$.
In fact, if 
\eqref{condition}
is describable as a non-trivial system of polynomial inequations for $g$,
then 
$(G,V)$ has  the Algebraic Sard Property. 
Case (3) was
already proved in \cite{LLMV2} by using an equivalent technique.

Equation \eqref{condition} does not have solutions in the following cases:  
 free-nilpotent groups of rank $2$ and step  $\geq 5$, %$> 4$, 
 free-nilpotent groups of rank $3$ and step  $\geq4$,
 free-nilpotent groups of rank      $\geq4$  and step     $\geq3$.  Here Sard's property remains an open problem.
%$\mathcal X$
%Therefore,
%if the set $\mathcal E$ is not empty then the abnormal set is a proper subset of $G$. Moreover, observing that $\mathcal E$ is defined by a polynomial relation, we can deduce that, whenever $\mathcal E$ is not empty then $G$ has  the (algebraic) Sard property. 

We further provide a more quantitative version of Sard's property for free-nilpotent groups of  step $2$.
\begin{theorem}\label{thm:free_step_two}
In any free-nilpotent group of  step $2$ 
the abnormal set is contained in an affine algebraic subvariety of codimension $3$.
\end{theorem}

%Regarding  Sard-type properties for step 2 Carnot groups, % i.e., the study of the bigger set $\Abn_{LM}(e)$, 
%we   mention 
%that  
 Agrachev, Lerario, and Gentile  
 previously 
  proved that in a  {\em generic}
Carnot group of  step $2$  
the {\em generic} point in the second layer is not in the abnormal set, see \cite[Theorem 9]{Agrachev_Lerario_Gentile}.

%satisfies the Sard Property. 
%\footnote{which version?}

There are several  papers
that give a bound on the size of 
the set 
%$\Abn_{LM}(e)$ 
of all those points $\End(u)$ where $u$ is a critical point with the extra property that $\gamma_u$ is
%the support
%abnormal curves that are also
 {\em length minimizing} for a fixed sub-Riemannian structure.
 A very general result \cite{Agrachev09} by Agrachev based on techniques of 
  Rifford and Tr\'elat 
  \cite{Rifford_Trelat} 
  %, MorseÐSard type results in sub-Riemannian geometry. Math. Ann.,}
  states that this set 
  % $\Abn_{LM}(e)$  
  is   contained in a closed nowhere dense set, for general sub-Riemannian manifolds.

 In this direction, in step 3 Carnot groups equipped with a sub-Riemannian structure on the first layer, we bound the size of the  
set $\Abn^{lm}(e)$ of points connected to the origin by locally length minimizing abnormal curves. 
Our result uses ideas of  Tan  and Yang 
 \cite{Tan_Yang_step_3} and the fact that 
in an arbitrary polarized Lie group the Sard Property holds for normal-abnormal curves, see Lemma~\ref{lemma:abn:nor}.

\begin{theorem}\label{thm:step_three}

 Let $G$ be a  sub-Riemannian Carnot group  of step 3. 
 The Sub-analytic Sard Property holds for locally length minimizing abnormal curves. Namely, the set $\Abn^{lm}(e)$ is contained in a sub-analytic set of codimension at least 1.
 
 \end{theorem}

The paper is organized as follows.
Section 2 is a preliminary section.  First we recall the definition of the endpoint map and we give a  characterization of the image of its differential in Proposition~\ref{primo-lemma}, in the case of polarized groups.
Secondly, we review Carnot groups, abnormal curves, and give  interpretations of the abnormal equations using left-invariant forms and right-invariant forms.
In Section~\ref{sec:abn:var}, we examine the notion of abnormal varieties.
In Section~\ref{subsection:normal} we review normal curves, and in Section~\ref{subsection:Goh} 
we review the
Goh condition.
In Section~\ref{sec:step:two} we consider step-$2$ Carnot groups. We first  prove the Algebraic  Sard Property for general Carnot groups of step $2$  and then we prove  Theorem \ref{thm:free_step_two} for free step-2 groups. For the latter,  we also give 
%more
 precise characterizations of the abnormal set.
In Section~\ref{sec:sufficient} we discuss 
sufficient conditions  for  Sard's property to hold. In particular, we discuss the role of contact vector fields and equation
\eqref{condition}.
The most important criteria are
Proposition~\ref{criterion_1} and
Corollary~\ref{criterium}, which will be used in Section~\ref{sec:applications} to prove the remaining part of  Theorem~\ref{main_thm1}.
In Section~\ref{sec:semi} we discuss Sard Property for a large class of semidirect  products of polarized groups.
In particular, we provide examples of  groups with exponential growth having the Analytic Sard Property (semisimple Lie groups) and the Algebraic Sard Property (solvable Lie groups). See Proposition~\ref{prop:semidirect} 
%for semidirect products 
and Remark~\ref{rem:semidirect}.
Section~\ref{sec:step:three}
is devoted to 
Carnot groups of step 3.
First we prove 
Sard Property for abnormal length minimizers, i.e., Theorem~\ref{thm:step_three}. Second, we investigate  the example of  
the free $3$-step rank-$3$ Carnot group, showing that the  argument used  in step-$2$ Carnot groups finds an obstruction:
there are abnormal curves  not contained  in any proper subgroup.  
We conclude the article with Section~\ref{sec:open}, where we discuss the open problems.

\noindent{\it Acknowledgments}
%The authors would like to thank E. Breuillard, U. Boscain, G. Citti, P. Koskela, U. Hamenst\"adt,   F. Jean, and A. Ottazzi for useful feedback and advice. 
Most of the work in this paper was developed while the authors were guests of the program {\it Geometry, Analysis and Dynamics on Sub-Riemannian Manifolds} at the Institut Henri Poincar\'e in the Fall 2014. 
The authors are very grateful to the program organizers 
A.~Agrachev,
D.~Barilari, 
U.~Boscain,
Y.~Chitour,
F.~Jean,
L.~Rifford, and
M.~Sigalotti,
  as well to IHP for its support. 
%In particular, we thank Breuillard and Ottazzi %, and Agrachev for discussing/explore/analyse 
%for discussing with us  Corollary~\ref{horizontal:differential}.

%%%%%%%%%%%%%%%%%%%%%%%%%%%%%%%%%%%%
%%%%%%%%%%%%%%%%%%%%%%%%%%%%%%%%%%%%

\section{Preliminaries}   	 
 
Let $G$ be a connected Lie group with Lie algebra $\g$, viewed as the tangent space of $G$ at the identity  element $e$. 
%For $u \in \g$, we denote by $X_u$ the left invariant vector field on $G$ which coincides with $u$ at the identity, and set $X_i:=X_{e_i}$.  
%We endow the vector space $\g$ with the Euclidean norm making
%$\{e_1,...,e_d\}$ an orthonormal basis.
% By a \emph{control}, we will mean an
For all $g\in G$, denote by $L_g$ and $R_g$ the left and right multiplication by $g$, respectively.
Also, $\Ad_g:= {\rm d}(L_g\circ R_{g^{-1}})_e  .$

%For all $X\in \g$ and $g\in G$, we will use the notation
%$$g\cdot X :=  (\dd L_{g})_e X     		.$$

Fix a linear subspace $V\subseteq\g$.
 Let $u$ be an element  of $L^2([0,1],V)$. 
 Denote by $\gamma_u$ the curve  in $G$ that solves the ODE:
\begin{equation}\label{ODE}
\frac{\dd\gamma}{\dd t}(t)=\left(\dd L_{\gamma(t)}\right)_e  u(t), %\gamma(t) \cdot u(t)
\end{equation}
with initial condition $\gamma(0)=e$. 
Viceversa, if 
$\gamma:[0,1]\to G$ is an absolutely continuous curve 
that solves   \eqref{ODE} for some 
$u\in L^2([0,1],V)$,
then we say that $\gamma$ is
 {\em horizontal} with respect to $V$  and that $u=u_\gamma$ is its {\em control}. In other words, the derivatives of $\gamma$ lie in the
  left-invariant subbundle,  denoted by $\Delta$, that  coincides with $V$ at $e$.

 The {\em endpoint map} starting at $e$ with controls in $V$ is the map 
\begin{eqnarray*}
\End: L^2([0,1],V)& \rightarrow &G\\
u\qquad&\mapsto& \gamma_u(1).
 \end{eqnarray*}
% that assigns to every     $u\in L^2([0,1],\g)$ the  point $\gamma_u(1)$. % of its lift $\{\gamma(t)\}_{t\in [0,1]}$.

\subsection{Differential of the endpoint map}   	\label{EndpointMap}			
 
%\subsection{The differential of the endpoint map}

 The following result is standard and a proof of it can be found (in the more general context of Carnot-Carath\'eodory manifolds) in      \cite[Proposition~5.2.5, see also Appendix~E]{Montgomery}.
 %, see also Appendix E.

\begin{theorem}[Differential of End]\label{endpoint} The endpoint map $\End$ is a smooth map between the Hilbert space $L^2([0,1],V)$ and $G$.
If $\gamma$ is a  horizontal curve 
leaving from the origin with control $u$, then
the differential of  $\End$ at $u$, 
% and its differential at the point $u$ 
  which is a map from $L^2([0,1],V)$ to the tangent space of $G$ at $\gamma(1)$,
  is given by
 \begin{eqnarray*}
 \dd\End_u v=( \dd  R_{\gamma(1)} )_e \int_0^1 \Ad_{\gamma(t)} v(t )\dd t  ,\qquad \forall v\in L^2([0,1],V) .
  \end{eqnarray*}
\end{theorem}

\begin{proof}[Sketch of the proof]  
The proof of a more general result can be found in \cite{Montgomery}.
We sketch here the simple proof of the formula in the case when $G \subset GL_n(\R)$, 
%for simplicity
% and refer to Montgomery for the general case. Then 
 where we can interpret the Lie product as a matrix product and work in the matrix coordinates. 
 Let  $\gamma_{u+\eps v}$ be the curve with the control $u+\eps v$ and $\sigma(t)$ be the derivative of $\gamma_{u+\eps v}(t)$ 
 with respect to $\eps$ at $\eps=0$. Then $\sigma$ satisfies the following ODE (which is the derivation with respect to $\eps$ of   $(\ref{ODE})$ for $\gamma_{u+\eps v}$)
 \begin{eqnarray*}
 \frac{\dd\sigma}{\dd t} = \gamma(t) \cdot v(t) + \sigma \cdot u(t).
  \end{eqnarray*}
Now it is easy to see that $\int_0^t \Ad_{\gamma(s)} ( v(s)) \dd s \cdot \gamma(t)$ satisfies the above equation with the same initial condition as $\sigma$, hence is equal to $\sigma$.  \end{proof}

\begin{proposition}[Image of $\dd \End$]\label{primo-lemma}
If $\gamma:[0,1]\to G$ is a  horizontal curve 
leaving from the origin 
with control $u $,
then 
%Then  the image of the differential at $u$ of  the endpoint map is 
% the right translate by $\gamma(1)$ of the subspace of the Lie algebra that  is the sum over all $t \in [0,1]$ of $\Ad_{\gamma(t)}V_1$, i.e.,
 \begin{eqnarray}\label{for:Image:of:dEnd}	
 {\rm Im }( \dd\End_u  ) =( \dd  R_{\gamma(1)} )_e (\span\{  \Ad_{\gamma(t)}V\;:\; t \in [0,1] \}). \end{eqnarray}
\end{proposition}

\proof
%\proof[Proof of Proposition~\ref{primo-lemma}] 
A glance at  the formula of Theorem~\ref{endpoint} combined with the fact that 
 $(\dd  R_{\gamma(1)} )_e$ is a linear isomorphism from $\g$ to $T_{\gamma(1)}G$ shows that it suffices to prove that  
 \begin{eqnarray*} \left\{ \int_0^1 \Ad_{\gamma(t)} v(t )\dd t :  v\in L^2([0,1],V) \right \} =  \span\{  \Ad_{\gamma(t)}V\;:\; t \in [0,1] \}. \end{eqnarray*}

$\subset:$ Any linear combination of terms $\Ad_{\gamma(t_i)} v_i$ is in the right hand set. Now an integral is a limit of finite sums 
and the right hand side is closed. Hence the right hand side contains the left hand side.

$\supset:$ It suffices to show that any element of the
form   $\xi = \Ad_{\gamma(t_1)} v_1$  lies in the left hand side.  Let $\psi_n (t)$ be a delta-function family centered at $t_1$,
that is, a  smooth family of continuous functions for which the limit as a distribution as $n \to \infty$
of $\psi_n (t)$ is $\delta(t -t_1)$.  Then $\lim_{n\to \infty} \int_0^1 \Ad_{\gamma(t)} \psi_n (t) v_1 \dd t =  \Ad_{\gamma(t_1)} v_1 = \xi$
and since the left hand side is a closed subspace,  $\xi$ lies in the set in the left hand side. 
 \qed

\begin{remark}
Evaluating \eqref{for:Image:of:dEnd} at $t=0$ and $t=1$ yields
\begin{equation}\label{R_V_L_V}
(\dd R_{\gamma(1)})_eV+(\dd L_{\gamma(1)})_e V 
\subset {\rm Im}(\dd \End_u).
\end{equation}
	
\end{remark}

\begin{remark}
Proposition~\ref{primo-lemma} implies immediately that for strongly bracket generating distributions, the endpoint map is a submersion at every $u\neq 0$.
We recall that a polarized group $(G,V)$ is  {\em strongly bracket generating}  if for every $X\in V\setminus\{0\}$, one has $V+[X,V]=\g$.
\end{remark}

\begin{remark}[Goh's condition is automatic in rank $2$]\label{rem:Gohrank2}
Assume that $\dim V = 2$.
We claim that if $\gamma$ is horizontal leaving from the origin with control $u$,
then 
for all $ t \in [0,1]$ we have
 \begin{equation}\label{eq:Gohrank2}
   ( \dd  R_{\gamma(1)} )_e  \Ad_{\gamma(t)} [V,V]
\subseteq
{\rm Im }( \dd\End_u  ) . 
\end{equation}
Indeed, we may assume that $\gamma$ is parametrized by arc length and that $t$ is a point of differentiability.
Hence,
$\gamma(t)^{-1}\gamma(t+\eps)= 
\exp(u(t) \eps + o(\eps))$.
Notice that since $u(t)\in V\setminus \{0\}$ and  $\dim V=2$, it follows that
$[u(t), V]=[V,V]$.
Therefore 
%\begin{eqnarray}
$\Ad_{\gamma(t)}^{-1}\Ad_{\gamma(t+\eps)}V 
%&=&
%\Ad_{\gamma(t)^{-1} \gamma(t+\eps)}V\\
%&=&
=
e^{   \ad_{u(t) \eps + o(\eps)}}V.
$
%\end{eqnarray}
Hence, for all $Y\in V$ 
%there exists 
 \begin{eqnarray*}
\eps[u(t), Y] + o(\eps)\in
V+\Ad_{\gamma(t)}^{-1}\Ad_{\gamma(t+\eps)}V. 
 \end{eqnarray*}
Therefore,   Proposition~\ref{primo-lemma} implies that
% for all $Y\in V$
$
  \Ad_{\gamma(t)}
  [u(t), Y]\in 
%  \subseteq
( \dd  R_{\gamma(1)} )_e^{-1}
{\rm Im }( \dd\End_u  )$, which proves the claim.

 By \eqref{eq:toGoh} below, formula \eqref{eq:Gohrank2} implies that, whenever $\gamma$ is an abnormal curve (see Section~\ref{sec:abnormal}) in a polarized group $(G,V)$ of rank 2, then $\gamma$ satisfies the {\em Goh condition} (see Section~\ref{subsection:Goh}).

\end{remark}

\begin{remark}[Action of contact maps]
We associate to the subspace $V\subseteq \g$  a left-invariant subbundle $\Delta$ of $TG$ such that 
$\Delta_e=V$. A vector field $\xi\in {\rm Vec}(G)$ is said to be {\em contact} if  its flow
$\Phi^s_\xi$  preserves $\Delta$.
 Denote by 
 \begin{eqnarray*}
\mathcal S := \{\xi\in {\rm Vec}(G) \mid \xi \text{ contact}, \xi_e=0\} \end{eqnarray*}
  the space of global contact vector fields on $G$ that vanish at the identity.
We claim that, for every horizontal curve $\gamma$ leaving from the origin,
\begin{equation}\label{S1}
\mathcal S(\gamma(1))\subset {\rm Im}(\dd \End_u).
\end{equation}
Indeed, let $\xi\in \mathcal S$ and let $\phi_\xi^s$ be the corresponding flow at time $s$. Since $\xi_e=0$, we have that $\phi_\xi^s(e)=e$.
 Consider the curve $\gamma^s:=\phi_\xi^s\circ \gamma$.
 Notice that 
 $\gamma^s(e)=e$
 and that
  $\gamma^s$ is horizontal,
 because $\xi$ is a contact vector field.
 %, each map 
% $\phi_\xi^s$
% preserves the subbundle.
%and $\gamma(e)=e$
Therefore, 
 \begin{eqnarray*}\End(u^s)=\gamma^s(1)= \Phi^s_\xi(\gamma(1)), \end{eqnarray*}
where $u^s$ is the control of
  $\gamma^s$.
Differentiating at $s=0$, we conclude that
$\xi(\gamma(1))$, which is an arbitrary point in 
$\mathcal S(\gamma(1))$, belongs to $ {\rm Im}(\dd \End_u).$
\end{remark}

%%%%%%%%%%%%%%%%%%%%%%%%%%%%%%%%%%%%

 \subsection{Carnot groups}   	\label{CarnotGroups}	 
Among the polarized groups, Carnot groups are the most distinguished.
A {\it Carnot group}  is a simply connected,    polarized Lie group $(G,V)$ whose
Lie algebra $\g$   admits % stratification %positively graded: 
a direct sum   decomposition  
%    $\mathfrak g=V_1\oplus V_2\oplus\cdots \oplus V_s$ 
  in     nontrivial  vector subspaces
 \begin{eqnarray*}\g = V_1 \oplus V_2 \oplus \ldots \oplus  V_s\quad \text{ such that} \qquad [V_i, V_j ] = V_{i+j} \end{eqnarray*}
where $ V_k =\{ 0\}, k > s$
and $V_1=V$.
%$V=V_1$.
% We take the polarization $V$ to be the first layer $V_1$: 
 %\begin{eqnarray*}V_1 = V. \end{eqnarray*}
 We refer to the $i$th summand $V_i$ as the $i$th {\it layer}.
% which we always assume generates the Lie algebra and therefore $[V_i, V_j ] = V_{i+j}$.  

The above decomposition   is also called the {\it stratification} of $\g$ and Carnot groups are often referred
to in the analysis literature as  {\it stratified} groups. 
  The {\it step} of a Carnot group is the total number $s$ of layers and  equals
  the degree of nilpotency of $\g$:  all Lie brackets of length greater than $s$ vanish.  
  Every Carnot group admits at least a canonical outer automorphism,   the `scaling' $\delta_\lambda$ which on $\g$ is equal to the multiplication by  $\lambda^i$ on the $i$th layer.
  
Since $G$ is  simply connected  and nilpotent, the exponential map $\exp: \g \to G$ is a diffeomorphism.
We write $\log$ for the inverse of $\exp$. 
When we  use $\log$  to  identify $\g$ with $G$ the group law  on $G$   becomes
a polynomial map $\g \times \g \to \g$ with $0 \in \g$ playing the role of the identity element $e \in G$.

%%%%%%%%%%%%%%%%%%%%%%%%%%%%%%%%%%%%

\subsection{Abnormal curves}   \label{sec:abnormal}

\begin{definition}[Abnormal curve]
Let $(G, V)$ be a polarized group.
Let $\gamma:[0,1]\to G$ be a horizontal curve leaving from the origin with control $u$.
If  
$ {\rm Im}(\dd \End_u)\subsetneq T_{\gamma(1)} G$,
%$\dd \End_u$ is not surjective,
we say that $\gamma$ is {\em abnormal}.
\end{definition}

In other words,
$\gamma$ is   abnormal 
%if and only if
%$ {\rm Im}(\dd \End_u)\subsetneq T_{\gamma(1)} G$
%if and only if
%$ u$ is a critical point of 
%$\End$
if and only if
$ \gamma(1)$ is a critical value of 
$\End$.
We define the {\em abnormal set} of $(G,V)$ as
%\begin{definition}
%critical values of the endpoint map
\begin{equation}\label{abnormal:set}
\Abn(e):=\{ \gamma(1)\mid \gamma \text{ abnormal }, \gamma(0)=e\}
=\{\text{critical values of } \End\}.
\end{equation}
%\end{definition}
The Sard Problem in sub-Riemannian geometry is the study of the above abnormal set. More information can be found in
  \cite[page 182]{Montgomery}.

\subsubsection{Interpretation of abnormal equations via right-invariant forms}   				
%	\todo{To do: Richard. Hamilt'n persp. }
Proposition~\ref{primo-lemma} gives an  interpretation for a curve to be abnormal,
which, to the best of our knowledge, is not in the literature.
			 
\begin{corollary}
%[of Proposition~\ref{primo-lemma}]
\label{right_invariant_interpretation}
Let $(G,V)$ be a polarized group and let $\gamma : [0,1]\to G$ be a horizontal curve.
Then the following are equivalent:
\begin{enumerate}
 \item $\gamma$ is abnormal;
  \item there exists $\lambda \in \g^*\setminus{\{0\}}$ such that  
$  \lambda(\Ad_{\gamma(t)}V )=\{0\}$ for every $t\in [0,1]$;
\item there exists a right-invariant 1-form $\alpha$ on $G$ such that
$  \alpha( \Delta_{\gamma(t)} )=\{0\}$ for every $t\in [0,1]$,
where $\Delta$ is the left-invariant distribution induced by $V$.
%\footnote{this is the third time that we define $\Delta$.}
\end{enumerate}
\end{corollary}			
\begin{proof}
$(2)$ and $(3)$ are obviously equivalent.
By Proposition~\ref{primo-lemma}, $\gamma$ is abnormal if and only if there is a proper subspace of $\g$
that contains $ \Ad_{\gamma(t)}V$ for all $t$.
\end{proof}

\subsubsection{Interpretation of abnormal equations via left-invariant adjoint equations}

The previous section  characterized   singular curves for a left-invariant distribution on a Lie group $G$  in terms of right-invariant one-forms.
This section characterizes the same   curves in terms of  left-invariant one-forms.
 This left-invariant characterization   is the one used in  \cite[Equations (12), (13) and (14)]{Montgomery_Singular_extremals_on_Lie_groups}    and  \cite[equations in Section 2.3]{Gole_Karidi}.
    We establish  the equivalence of the   two characterizations  directly using  Lie theory.  
Then we take a second, Hamiltonian, perspective on the  equivalence of characterizations.    In this perspective,  
the right-invariant characterization is simply  
the momentum map applied to  the Hamiltonian provided  by  
%\footnote{Richard, is it correct this double proposition?}
the Maximum Principle.

We shall also introduce the notation  
 \begin{equation}
 \label{curv_map}
 w(\eta) (X,Y):=\eta([X,Y]), \text{ for }\eta \in V^{\perp} \subset \g^*, X,Y\in V.
 \end{equation}
 
\begin{proposition}
\label{left_invariant_interpretation}
Let $(G,V)$ be a polarized group and let $\gamma : [0,1]\to G$ be a horizontal curve with control $u$.
Then the following are equivalent:
\begin{enumerate}
 \item $\gamma$ is abnormal; 
  \item there exists a curve $\eta: [0,1] \to  \g^*$, with  $\eta(t) |_V =0$ and  $\eta(t) \ne 0$, for all $t\in [0,1]$,
 % \footnote{either we put $ \g^*\setminus{\{0\}}$ or $\eta(t) \ne 0$}
  representing a curve of left-invariant one-forms, 
   such that  
\begin{equation*}
\label{invariant_left_eqns}	
\left\{\begin{array}{l}
\frac{\dd \eta }{\dd t}(t) = (\ad_{u(t)}) ^* \eta(t)\\
u(t)\in {\rm Ker} (w(\eta(t))).
\end{array}\right.
\end{equation*}
%In (2) the curve $\eta(t)$ represents a curve of left-invariant one-forms.
\end{enumerate}
\end{proposition}

\begin{remark} There is a  sign difference between the first equation of (2) above,  namely $\frac{\dd \eta }{\dd t}(t) = (\ad_{u(t)}) ^* \eta(t)$,
%characterizing the abnormals,
 and the analogous equation in \cite[Sec.~4]{Montgomery_Singular_extremals_on_Lie_groups} 
that reads 
$ \frac{\dd \eta }{\dd t}(t) = -\ad_{u(t)} ^* \eta(t)$.  
The equations  coincide if we set $\ad_u ^* = -(\ad_u)^*$.  To understand this minus sign,
we first observe that 
in the equation above    $(\ad_u)^* $ is the operator $(\ad_u)^*: \g^* \to \g^*$  dual to the adjoint operator, so that
 \begin{eqnarray*}
 ((\ad_u)^* \lambda )(X) = \lambda (\ad_u (X)) = \lambda([u,X]). \end{eqnarray*}
In the equation of \cite[Sec.~4]{Montgomery_Singular_extremals_on_Lie_groups}
the operator $\ad_u ^*$ is the differential of the co-adjoint action $\Ad^*: G \to gl(\g^*)$
taken at $g = e$ in the direction $u \in \g$. The minus sign arises out of the inverse needed to make the action a left action:  
 $\Ad^* (g) = (\Ad_{g^{-1}})^*$.
 % and this inverse is required to make the action of $G$ on $\g^*$ a left action.
\end{remark}

Gol\'e and Karidi   made good use of the coordinate version of the previous proposition.   
See \cite[page 540]{Gole_Karidi}, following \cite[Sec.~4]{Montgomery_Singular_extremals_on_Lie_groups}. See also \cite{LLMV, LLMV2}.
To describe their version, fix  a basis $X_1, \ldots, X_n$ of $\g$ 
such that $X_1, \ldots, X_r$ is a basis of $V$.
 Let $c_{ij}^k$ be the structure constant of $\g$ with respect to this basis, seen as {\it left}-invariant vector fields.  Let $(u_1, \ldots, u_r) \in V$
 be controls relative to this basis. Let $\eta_i = \eta (X_i)$ denote the linear coordinates of a covector $\eta \in \g^*$
 relative to this basis.  

\begin{proposition}
\label{Gole_Karidi_form}
Let $(G,V)$ be a polarized group.
Let $\gamma : [0,1]\to G$ be a horizontal curve with control $\sum_{i=1} ^r u_i (t) X_i$.
Under the above coordinate conventions,  the following are equivalent:
\begin{enumerate}
 \item $\gamma$ is abnormal;
 \item 
 there exists a vector function  $(0,0, \ldots, 0, \eta_{r+1}  
 %\eta_{1}
 , \ldots, \eta_n):[0,1] \to \R^n$, never vanishing,  
 such that
 \begin{equation*}	
\left\{\begin{array}{l}
%\eta_{1}= \ldots= \eta_r=0\\
%(\eta_{r+1}(t), \ldots, \eta_n(t))\neq (0,\ldots, 0), \text{ for all } t\in [0,1],\\
\frac{\dd \eta_i }{\dd t}(t) + \sum_{j=1}^r \sum_{k=r+1}^n
c_{ij}^k  u_j(t) \eta_k(t)  =  0  ,\qquad \text{ for all } i=r+1,\ldots, n, \\
 \sum_{j=1}^r \sum_{k=r+1}^n
c_{ij}^k  u_j(t) \eta_k(t)  =  0  ,\qquad \text{ for all } i=1,\ldots, r.
\end{array}\right.
\end{equation*}
% s a right-invariant 1-form $\alpha$ on $G$ such that
%$  \alpha( \Delta_{\gamma(t)} )=\{0\}$ for every $t\in [0,1]$,
%where $\Delta$ is the left-invariant distribution induced by $V$.
\end{enumerate}
\end{proposition}

Both Corollary~\ref{right_invariant_interpretation} and Proposition~\ref{left_invariant_interpretation}  lead to 
a one-form $\lambda(t) \in T^* _{\gamma(t)} G$   along the curve $\gamma$ in $G$.
The key to the equivalence of the right and left perspectives of these two propositions is that these one-forms along $\gamma$ are {\it equal}. 
For the right-invariant version, Corollary~\ref{right_invariant_interpretation}  provides first the constant covector $\lambda^R \in \g^* = T^* _e G$, and then its {\it right}-invariant
extension. Finally we evaluate this extension along $\gamma$. For the left-invariant version, following Proposition~\ref{left_invariant_interpretation},  we
take the curve of covectors $\eta(t)$,  consider their {\it left}-invariant extensions, say $\eta(t)^L$  (leading to a curve of left-invariant one-forms)
and finally  we   evaluate $\eta(t)^L$ at $\gamma(t)$.   The   following lemma  establishes
that the forms obtained in these two different ways coincide  along $\gamma$.

\begin{lemma} Let  $\gamma(t)$ be the curve in $G$ starting at $e$ and having  control $u(t)$. Let $\lambda(t)$ be a one-form defined along 
$\gamma$.  Let $\lambda ^R (t) =  (R_{\gamma(t)} )^* \lambda(t) \in \g^*$ be this one-form viewed by right-trivializing $T^*G$.
Let $\eta (t) = ( L_{\gamma(t)})^* \lambda(t) \in \g^*$ be this same one-form viewed by left-trivializing $T^*G$.
Then $\lambda^R (t)$ is constant if and only if $\eta(t)$ solves the time-dependent linear differential equation 
  $d \eta /dt = (\ad_{u(t)} )^* \eta(t)$ with initial condition $\eta(0) = \lambda(0)$.
\end{lemma}

Proof.  Suppose that $\lambda^R(t)$ is constant: $\lambda^R (t) \equiv \lambda^R$.  Set $g =\gamma(t)$. Then $\lambda(t) = ( R_{g}^{-1 })^* \lambda^R$ 
and consequently $\eta(t) =  (L_g )^*   (R_g ^{-1 })^* \lambda^R = (\Ad_g)^* \lambda^R$.  For small $\Delta t$ we write  $\gamma(t + \Delta t) = \gamma(t) (\gamma(t) ^{-1} \gamma(t + \Delta t) ) = g h$
with $h = h({\Delta t}) =  \gamma(t) ^{-1} \gamma(t + \Delta t)$ and use $(\Ad_{gh})^* = (\Ad_h) ^* (\Ad_g)^*$ to establish the identity for the difference quotient:
\[ \frac{1}{\Delta t} ( \eta(t + \Delta t) - \eta(t) ) = \frac{1}{\Delta t} ( (\Ad_{h(\Delta t)} )^* - {\rm Id}) \eta(t). \]
Now we use that the derivative of the adjoint representation $h \mapsto \Ad_h$
%with respect to $h$ and 
evaluated at the identity, is the standard adjoint representation $\g \to \mathfrak{gl}(\g)$ ,  $X \to \ad_X = [X, \cdot ]$.
Taking duals, we see that the difference quotient $\frac{1}{\Delta t} (( \Ad_{h(\Delta t)}) ^* - {\rm Id})$ 
limits to the linear operator  $(\ad_{u(t)}) ^*$ on $\g^*$.  
%(The derivative associated to $h(\Delta t)$ is   $u(t)$ since, using suggestive matrix   notation,  $\gamma(t + \Delta t) = \gamma(t) + \Delta t \gamma(t) u(t) $ so that $h(t) = \gamma(t) ^{-1} \gamma(t + \Delta t) ) = (e + \Delta t u(t) + O(\Delta t ^2)$. ) 
%Then $(Ad_h ^* - Id) \eta =Ad_{e + \Delta t u(t)} ^* \eta  - \eta + O(\Delta t ^2)$ which, upon differentiating, yields that $\dot \eta = ad_{u(t)} ^* \eta$.  
 
The steps just taken are reversed with little pain,  showing the equivalence.
 \qed
 
 \subsection{Hamiltonian formalism and reduction}
 \label{Ham_formalism}
 
We describe the Hamiltonian  perspective on Corollary~\ref{right_invariant_interpretation}, Proposition~\ref{left_invariant_interpretation} 
and the relation between them. 

 We continue   with the basis  $X_i$ of left-invariant vector fields on $G$,
labelled so that the first $r$ form a basis of $V$.  
Write $P_i : T^* Q \to \R$ for the same fields, but viewed as fiber-linear functions on the cotangent bundle of $G$:
\begin{equation}
\label{momentum_fns}
P_i: T^* G \to \R;  P_i (g,p) = p(X_i (g)).
\end{equation} 
 Given a choice of controls $u_a (t)$, $a =1 , 2 \ldots, r$  not all identically zero, form the Hamiltonian  
% \footnote{Change of notation for the Hamiltonian: $\H$. really necessary?}
 \[H_u (g,p; t) = \sum_{i = 1} ^r u_a (t) P_a (g,p).\]
 The Maximum Principle  
  \cite[Theorem 12.1]{Agrachev_Sachkov}    asserts
    that a curve $\gamma$ in $G$ is singular for $V$ if and only if 
 when we take its control $u$, and form the Hamiltonian   $H_u$,
 then the corresponding Hamilton's equations  have a  nonzero solution $\zeta (t) = (q(t), p(t))$ that lies on the variety
$P_a = 0, a = 1, 2, \ldots, r$.    Here `Nonzero' means that   $p(t) \ne 0$, for all $t$.  
 The conditions $P_a = 0$ mean  that 
the solution lies in the  annihilator of  the distribution defined by $V$.  The first of Hamilton's equations, implies that $\gamma$
has control $u$, so that  the solution $\zeta$ does project onto  $\gamma$
 via the cotangent projection  $\pi: T^*G\to G$.

The following two facts regarding symplectic geometry and Hamilton's equations allow us to immediately
derive the Gol\'e-Karidi form of the equations as expressed in Proposition~\ref{Gole_Karidi_form}.  Fact 1.
Hamilton's equations   are equivalent to  their `Poisson form'
$\dot f = \{f, H \}$.  Here $f$ is an arbitrary smooth function on phase space,   $\dot f = df(X_H)$ is the derivative  of $f$ along the Hamiltonian
vector field $X_H$ for $H$, and   $\{f , g\}$ is the Poisson bracket associated to the canonical symplectic form
$\omega$, so that $\{f, g \} = \omega (X_f, X_g)$.
Fact 2. If $X$ is any  vector field on $G$ (invariant or not), and  if $P_X : T^*Q \to \R$ denotes the corresponding fiber-linear function defined by $X$ as above,  
then $\{P_X, P_Y \} = -P_{[X, Y]}.$ 

{\bf Proof of Proposition~\ref{Gole_Karidi_form} from the Maximum Principle.}
Take the $f = P_i$ and use, from Fact 2, that $\{P_i, P_j\} = -\sum c_{ij} ^k P_k$.
The $P_i$ are equal to the $\eta_i$ of the proposition.

Proposition~\ref{Gole_Karidi_form} is just the coordinate form of Proposition~\ref{left_invariant_interpretation},
so we have also proved Proposition~\ref{left_invariant_interpretation}.

{\bf Proof of Corollary~\ref{right_invariant_interpretation} from the Maximum principle.}

Let $\gamma(t)$ be a singular extremal leaving the identity with control $u = (u_1, \ldots , u_r)$.
Let $H_u$ be the time-dependent Hamiltonian generating the one-form $\zeta(t)$ along $\gamma$
as per the Maximum Principle. Since 
each of the $P_i$ are left-invariant, so is  $H_u$. 
Now any left-invariant  Hamiltonian  $H_u$ on the cotangent bundle of a Lie group
 admits $n = \dim(G)$ `constants' of motion -- these being the $n$ components of the momentum map
$J: T^*G \to \g^*$ for the action of $G$ on itself by {\it left} translation. Recall that a `constant of the motion' is a vector function that
is constant along all the solutions to Hamilton's equations.  Different solutions may have different constants. 
The  momentum map in this situation  is well-known to  
 equal   {\it right-}trivialization: $T^*G \to G \times \g^*$ composed  with projection onto the second factor.
In other words, if $\zeta (t)$ is any solution for $H_u$, then $J(\zeta(t)) = \lambda = {\rm const}$ and also
$J(\zeta(t)) = \dd R_{\gamma(t)} ^* \zeta(t)$.
  Now,  our $p(t)$ must annihilate $V_{\gamma(t)}$.
The fact that $p(t)$ equals $\lambda$, right-translated along $\gamma$, and that $\Delta_{\gamma(t)}$
equals to $V = \Delta_e$, {\it left-translated} along $\gamma$ implies that $\lambda(\Ad_{\gamma(t)} V) = 0$.
We have established the claim.
\qed

\subsection{Abnormal varieties and connection with extremal polynomials}\label{sec:abn:var}
The opportunity of considering the right-invariant trivialization of $T^*G$, hence arriving to
Corollary~\ref{right_invariant_interpretation}, was suggested by the results of the two papers \cite{LLMV, LLMV2}, where abnormal curves were characterized as those horizontal curves lying in specific algebraic varieties.

Given $\lambda \in \g^*\setminus{\{0\}}$ we set 
\begin{eqnarray}\label{abnormal_variety}
Z^{\lambda}:=\{g\in G\,:\,( (\Ad_{g})^*\lambda)_{|V }=0 \}.
\end{eqnarray}
In every Lie group the set $Z^{\lambda}$
is a proper real analytic variety.
If $G$ is a nilpotent group, then  $Z^{\lambda}$ is a proper real algebraic variety, which we call {\em abnormal variety}.

\begin{proposition}[Restatement of Corollary~\ref{right_invariant_interpretation}]\label{polynomial}
A horizontal curve $\gamma$ is abnormal if and only if
%there exists a nonzero $\alpha\in\g^*$ such that
 $\gamma$ is contained in  $Z^{\lambda}$
 for some nonzero $\lambda \in \g^*$.
  \end{proposition}

We now prove that, in the context of Carnot groups, the algebraic varieties $Z^{\lambda}$ coincide with the varieties introduced in the papers \cite{LLMV,LLMV2}. This will follow from Proposition~\ref{prop:abnvar} below.

Let $e_1,\ldots,e_n$ be a basis of $\g$ such that $e_1,\ldots,e_r$ is a basis of $V$. Let $X_i$ denote the extension of $e_i$ as a left-invariant vector field on $G$. Let $c_{ij}^{k}$ be the structure constants of $\g$ in this basis, i.e.,
\begin{eqnarray*}
[X_i,X_j]=\sum_{k}c_{ij}^{k}X_k .
\end{eqnarray*}
For $\lambda\in\g^*$, set
\begin{eqnarray*}
P^{\lambda}_i(g):=  ((\Ad_{g})^*\lambda)(e_i).
\end{eqnarray*}
Thus $Z^\lambda$ is the set of common zeros of the functions $P^{\lambda}_i$, $i=1,\ldots,r$. When $G$ is nilpotent, these functions are polynomials.

\begin{proposition}\label{prop:abnvar}
Let $Y_m$ denote the extension of $e_m$ as a right-invariant vector field on $G$. Let $e^*_1,\ldots,e^*_n$ denote the  basis vectors of $\g^*$ dual to $e_1,\ldots,e_n$. For all $i,j=1,\dots,n$,
 we have
\begin{equation}\label{obs:Pierre}
X_i=\sum_m P^{e^*_m}_i Y_m . %\quad\forall\:i=1,\dots,n.
\end{equation}
Moreover, the functions $P^{\lambda}_j$ satisfy $P^{\lambda}_j(e)= \lambda(e_j)$ and
\begin{equation}\label{eq:Maya}
X_i P^{\lambda}_j =\sum_{k=1}^n c_{ij}^{k}P^{\lambda}_k, \quad\forall\:i,j=1,\dots,n, \lambda \in\g^*.
\end{equation}
In particular, in the setting of Carnot groups the functions $P^{\lambda}_j$ coincide with the {\em extremal polynomials} introduced in \cite{LLMV,LLMV2}.
\end{proposition}
\begin{proof}
We verify \eqref{obs:Pierre} by
\[
\sum_m P^{e^*_m}_i(g) Y_m(g) =\sum_m  (\Ad_{g})^*(e^*_m)(e_i) (R_g)_* e_m 
=\sum_m  e^*_m(\Ad_{g}(e_i)) (R_g)_* e_m 
\]
\[ 
=(R_g)_* \sum_m  e^*_m(\Ad_{g}(e_i))  e_m 
=(R_g)_* \Ad_{g}(e_i)
=(L_g)_*  e_i 
=X_i(g).
% =(L_g)_* (\Ad_{g})_* e_i
\]
Next, on the one hand,  since $[X_i,Y_j]=0$,
\begin{eqnarray*}
[X_i,X_j]=\sum_m (X_i P^{e^*_m}_j) Y_m.
\end{eqnarray*}
On the other hand, from \eqref{obs:Pierre}
\[
[X_i,X_j]=\sum_k c^k_{ij} X_k 
=\sum_m (\sum_k c^k_{ij} P^{e^*_m}_k) Y_m.
\]
Thus
\begin{eqnarray*}
X_i P^{e^*_m}_j = \sum_k c^k_{ij} P^{e^*_m}_k,\qquad\forall\:i,j,m=1,\dots,n. 
\end{eqnarray*}
Formula \eqref{eq:Maya} follows because, by definition, the functions $P_j^\lambda$ are  linear in $\lambda$.

The extremal polynomials $(P_j^v)_{j=1,\dots,n}^{v\in\R^n}$ were introduced in \cite{LLMV,LLMV2} in the setting of Carnot groups; they were explicitly defined in a system of exponential coordinates of the second type associated to a basis of
 $\mathfrak g$ that is
 adapted
to the stratification of $\g$, see Section~\ref{CarnotGroups}.
 Here, {\em adapted} simply means that the fixed basis $e_1,\ldots,e_n$  of $\g$ consists of an (ordered) enumeration of a basis of the first layer $V_1$, followed by a basis of the second layer $V_2$, etc.
 It was proved in \cite{LLMV2} that the extremal polynomials satisfy 
\[
P_j^v(e)=v_j\quad\text{and}\quad X_i P^{v}_j =\sum_{k=1}^n c_{ij}^{k}P^{v}_k\qquad\forall\:i,j=1,\dots,n,\forall\:v\in\R^n.
\]
We need to check that, for any fixed $v\in\R^n$,
the equality  $P_j^v=P_j^\lambda$ holds for $\lambda:=\sum_m v_m e_m^*$. Indeed, the differences $Q_j:=P_j^v-P_j^\lambda$ satisfy
\[
Q_j(e)=0\quad\text{and}\quad X_i Q_j =\sum_{k=1}^n c_{ij}^{k}Q_k\qquad\forall\:i,j=1,\dots,n.
\]
In particular, $X_iQ_n=0$ for any $i$ because, by the stratification assumption, $c_{in}^k=0$ for any $i,k$. This implies that $Q_n$ is constant, i.e., that $Q_n\equiv 0$. We can then reason by reverse induction on $j$ and assume that $Q_k\equiv 0$ for any $k\geq j+1$; then, using the fact that $c_{ij}^{k}=0$ whenever $k\leq j$ (because the basis is adapted to the stratification), we have
\[
Q_j(e)=0\quad\text{and}\quad X_i Q_j =\sum_{k=j+1}^n c_{ij}^{k}Q_k=0\qquad\forall\:i=1,\dots,n.
\]
Hence also $Q_j\equiv 0$. This proves that $P_j^v=P_j^\lambda$, as desired.
%\todo{To conclude, this proof needs more references to \cite{LLMV2}}
\end{proof}

\begin{remark}\label{rem:W:linear}
In the study of Carnot groups of step 2 and step 3, it will be used that the varieties  $W^\lambda$  defined below (which coincide with the abnormal varieties in the step-2 case) are subgroups. 
%In the study of Carnot groups of  step 2 (respectively, step 3), it will be used that the abnormal varieties (respectively, their higher step analog) are subgroups. 
Namely, if $G$ is a Carnot group of step $s$ and highest layer $V_s$, and $\lambda\in \g^*$, then 
the variety 
\begin{equation}
W^\lambda :=\{ g \in G : ((\Ad_g)^*\lambda)_{|V_{s-1}} =0  \}
\end{equation}
is a subgroup, whenever it contains the origin.
Indeed,
if $X\in \g$ and $Y \in V_{s-1}$, then 
\begin{eqnarray*}
( \Ad_{\exp(X)}  )^*\lambda (Y)=(e^{\ad_X}  )^*\lambda (Y) = \lambda(Y +[X,Y]).
\end{eqnarray*}
%W^lambda =  exp
Hence, in exponential coordinates the set   $W^\lambda$ is
\[ \{ X \in \g  :   \lambda (Y + [X,Y]) = 0,\, \forall Y \in V_{s-1}  \}\]
and, if it contains the origin, it is 
\[ \{X \in \g  :    \lambda ( [X,Y]) = 0 , \forall Y \in V_{s-1} \} .\]
Since the condition $ \lambda ( [X,Y]) = 0$, for all $Y \in V_{s-1}$,  is linear in $X$,
we conclude that $W^\lambda$ is a subgroup.
\end{remark}

\subsection{Lifts of abnormal curves}   				 

%proposition about quotient (lift of abn is abn)

\begin{proposition}[Lifts of abnormal is abnormal]\label{lift_abnormals}
Let $\gamma:[0,1]\to G$ be a horizontal curve with respect to $V\subset \g$.
If there exists a Lie group $H$ and a surjective homomorphism $\pi: G\to H$ for which 
$\pi\circ \gamma $ is abnormal with respect to some $W\supseteq \dd \pi_e (V)$, 
then 
$\gamma$ is abnormal.
\end{proposition}
\begin{proof}
Let $\End^V$ and $\End^W$ be the respective endpoint maps, as in the diagram below.
For   $u\in L^2([0,1],V)$ 
let 
$\pi_*u:=  \dd \pi_e\circ u$, which is an element in $L^2([0,1],W)$, because $ \dd \pi_e (V)\subseteq W$. 
Since $\pi$ is a group homomorphism, one can easily check that the following diagram commutes:
\[ 
\xymatrix{
	L^2([0,1],V)
	\ar[rr]^{\qquad\End^V} 
		\ar[d]_{\pi_*} 
	& &
		G \ar[d]^{\pi} 
\\
	L^2([0,1],W) \ar[rr]^{\qquad\End^W}  	  
	&  
%	\ar@/^1.1pc/[ur]_{\phi_\beta}
	&  H.
		}
\]
By assumption $\pi$ is surjective and so is $\dd \pi_g$, for all $g\in G$.
We conclude that
$\dd \End^W_{\pi_*u}$ is surjective, whenever
$\dd \End^V_{u}$ is surjective.
\end{proof}

    \begin{example}[Abnormal curves in a product]\label{ex:product}
    Let $G$ and $H$ be two Lie groups. Let $V\subset \Lie(G) $ and $W\subset \Lie(H) $.
    Assume that
    $W\neq \Lie(H) $.
    Let   $\gamma:[0,1]\to G\times H$ be a  curve. %horizontal
    If
     $\gamma=(\gamma_1(t),e)$
    with $\gamma_1:[0,1]\to G$ horizontal with respect to $V$,
    then
    $\gamma$ is abnormal with respect to $V\times W$.
    Indeed, this fact is an immediate consequence  of Proposition~\ref{lift_abnormals} using the projection
    $G\times H\to H$ and the fact that the constant curve in $H$ is 
    abnormal with respect to the proper subspace $W$.
    \end{example}
    
\begin{remark}\label{rmk:product}
 Let $G$ and $H$ be two Lie groups. %Let $V\subset \Lie(G) $ and $W\subset \Lie(H) $.
    If
    $\gamma_1:[0,1]\to G$ is not abnormal with respect to some  $V\subset \Lie(G) $  and
    $\gamma_2:[0,1]\to H$ is not abnormal with respect to some $W\subset \Lie(H) $,
    then 
    $(\gamma_1,\gamma_2):[0,1]\to G\times H$ is not abnormal with respect to $V\times W$.
\end{remark}

     \begin{example}[$H\times H$]
     Let $H$ be the Heisenberg group equipped with its contact structure.
     By Example~\ref{ex:product} and Remark~\ref{rmk:product},
     the abnormal curves leaving from the origin in $H\times H$ are the curves of the form
     $(\gamma(t),e)$ or $(e,\gamma(t))$, where $\gamma:[0,1]\to H$ is {\em any} horizontal curve.
     In particular, $\Abn(e)=H\times\{e\} \cup \{e\}\times H$, which has codimension $3$.     
     \end{example}

\subsection{Normal curves}\label{subsection:normal}
Let $(G,V)$ be a polarized group such that $V$ is bracket generating. Equipping $V$ with a scalar product $\|\cdot \|_2$, we get a left-invariant sub-Riemannian structure on $G$.
Recall that from Pontrjagin Maximum Principle 
any curve that is
length minimizing with respect to the sub-Riemannian distance is either abnormal, or normal (in the sense that we now recall), or both normal and abnormal.
A curve $\gamma$ with control $u$ is {\em normal} if there exist $\lambda_0\not=0$ and $\lambda_1\in T_{\gamma(1)}^{*}G$ such that $(\lambda_0,\lambda_1)$ vanishes on the image of the differential at $u$ of the {\em extended endpoint map} $\widetilde{\End}:L^2([0,1],V)\to\R\times G$, $v\mapsto (\|v\|_2,\End(v))$.
   Let $\Abn^{nor}(e)$ denote the set of points connected to the origin by curves which are both normal and abnormal. 
        Let $\Abn^{lm}(e)$ denote the set of points connected to the origin by  abnormal curves
  that are locally length minimizing with respect to the sub-Riemannian distance.
 
\begin{lemma}\label{lemma:abn:nor}
Let $G$ be a polarized Lie group. The Sard Property holds for normal abnormals. Namely, the set $\Abn^{nor}(e)$ is contained in a sub-analytic set of codimension at least 1.
\end{lemma}

\begin{proof}
We will make use of the sub-Riemannian exponential map, see \cite{}.
Namely, normal curves starting from $e$ have cotangent lifts which satisfy a Hamiltonian equation. Solving this equation with initial datum $\xi\in T^*_e G$ defines a control $\widetilde{Exp}(\xi)\in L^2([0,1],V)$. Composing with the endpoint map, one gets the sub-Riemannian exponential map $Exp:T^*_e G\to G$,
\begin{eqnarray*}
Exp=\End\circ \widetilde{Exp}.
\end{eqnarray*}
Points in $\Abn^{nor}(e)$ are values of $Exp$ where the differential of $\End$ is not onto. Therefore, they are singular values of $Exp$. Since $Exp$ is analytic, the set of its singular points is analytic, thus the set of its singular values is a sub-analytic subset of $G$. By Sard's theorem, it has measure zero, therefore its codimension is at least 1.
\end{proof}

%Tan and Yang have proved that for step 3 Carnot groups, all abnormal length minimizing curves are normal. Here is a sketch. By Goh's theorem, if one such curve $\gamma$ is not normal, then it satisfies Goh's condition. This states that $\gamma$ is contained in an abnormal variety $Z^\lambda$ with $\lambda\in V_3^*$. Such $Z^\lambda$'s are subgroups. The accessible set $H^\lambda$ in $Z^\lambda$ is a proper Carnot subgroup. Since $\gamma$ is still length minimizing in $H^\lambda$, either $\gamma$ is normal, and we stop, or it is abnormal and not normal, and we iterate. Eventually, we obtain that $\gamma$ is normal or nonsingular within a Carnot subgroup.
%
%
%
%---

\subsection{The Goh condition}\label{subsection:Goh}
Let $(G,V)$ be a polarized group as in Section~\ref{subsection:normal}. We introduce the well-known  Goh condition by using the formalism of Corollary \ref{right_invariant_interpretation}.

\begin{definition}
We say that an abnormal curve $\gamma:[0,1]\to G$ leaving from the origin $e$ satisfies the Goh condition if there exists $\lambda\in \mathfrak g^\ast\setminus\{0\}$ such that
\begin{equation}\label{eq:defiGoh}
\lambda (\Ad_{\gamma(t)} (V + [V,V])) = 0\quad\text{for every }t\in[0,1].
\end{equation}
\end{definition}

Equivalently, $\gamma$ satisfies the Goh condition if and only if there exists a right-invariant 1-form $\alpha$ on $G$ such that $  \alpha( \Delta^2_{\gamma(t)} )=\{0\}$ for every $t\in [0,1]$, where $\Delta^2$ is the left-invariant distribution induced by $V + [V,V]$. Equivalently, denoting by $u$ the controls associated with $\gamma$ and recalling Proposition~\ref{primo-lemma}, if and only if the space
\begin{equation}\label{eq:toGoh}
\bigcup_{t\in[0,1]} \Ad_{\gamma(t)} (V + [V,V]) = dR_{\gamma(1)}^{-1} ({\rm Im }( \dd\End_u  )) + \bigcup_{t\in[0,1]}\Ad_{\gamma(t)} ([V,V])
\end{equation}
is a proper subspace of $\mathfrak g=T_eG$, which  a posteriori is contained in ker $\lambda$, for $\lambda$ as in \eqref{eq:defiGoh}. 

\begin{remark}\label{rem:anniGoh}
Clearly, any $\lambda$ such that \eqref{eq:defiGoh} holds is in the annihilator of $V+[V,V]$,
just by considering $t=0$ in \eqref{eq:defiGoh}. 
\end{remark}

The importance of the Goh condition stems from the following well-known fact: if $\gamma$ is a {\em strictly abnormal} length minimizer (i.e., a length minimizer that is abnormal but not also normal), then it satisfies Goh condition for some $\lambda\in \mathfrak g^\ast\setminus\{0\}$. See \cite[Chapter 20]{Agrachev_Sachkov} and also \cite{Agrachev_Sarychev}.   Notice that not necessarily all the $\lambda$'s as in (2) of Corollary \ref{right_invariant_interpretation}  will satisfy \eqref{eq:defiGoh}, but at least one will. On the contrary, in the particular case dim $V=2$,  
every abnormal curve satisfies the Goh condition for every $\lambda$  as in Corollary \ref{right_invariant_interpretation} (2); see Remark \ref{rem:Gohrank2} and  \eqref{eq:Gohrank2} in particular.

%%%%%%%%%%%%%%%%%%%%%%%%%%%%%%%%%%%%

\section{Step-$2$ Carnot groups}   		\label{sec:step:two}		 
\subsection{Facts about abnormal curves in   two-step Carnot groups}
\label{facts2step}
We want to study the abnormal set ${\rm Abn}(e)$ defined in \eqref{abnormal:set} with the use of the abnormal  varieties defined in \eqref{abnormal_variety}. In fact,
by Proposition~\ref{polynomial} we have the inclusion 
\begin{eqnarray*}
{\rm Abn}(e)\subseteq \bigcup_{\substack{\lambda \in \g^* \setminus \{0\}\text{ s.t. }e\in Z^\lambda}} Z^\lambda.
\end{eqnarray*}
In this section we will consider 
%We start with some preliminary considerations valid for any 
the case when the polarized group $(G, V)$ is   a Carnot group   of step 2. Namely,  
the  Lie algebra of $G$ admits the decomposition $\g=V_1\oplus V_2$ with $V=V_1$, $[V_1, V_1]=V_2$, and $[\g, V_2]=0$.
Fix an element $\lambda\in \g^*$. Since $\g^*=V_1^*\oplus V_2^*$, we can write    $\lambda=\lambda_1+\lambda_2$ with 
$\lambda_i\in V_i^*$.
As noticed in Remark \ref{rem:W:linear},
since $G$ has step $2$, if $X\in \g$ and $Y\in V_1$, then
\begin{eqnarray*}
( \Ad_{\exp(X)}  )^*\lambda (Y)=(e^{\ad_X}  )^*\lambda (Y) = \lambda_1(Y) + \lambda_2([X,Y]).
\end{eqnarray*}
Notice that, if  $e=\exp(0)\in Z^\lambda$, then $\lambda_1(Y) =0$ for all $Y\in V_1$. Thus $\lambda_1=0$.
 Therefore, any variety $Z^\lambda$ containing the identity is of the form
\begin{eqnarray*}
Z^\lambda=Z^{\lambda_2}=\exp\{X\in \g \;:\;    \lambda_2([X,Y]) = 0 \ \forall\:  Y\in V_1   \}. 
\end{eqnarray*}
The condition 
\begin{eqnarray*}
\lambda_2([X,Y]) = 0,\qquad \ \forall\:  Y\in V_1,
\end{eqnarray*}
is linear in $X$, hence the set
\begin{eqnarray*}
\mathfrak z^\lambda:= \log(Z^\lambda)=  \{X\in \g \;:\;    \lambda_2([X,Y]) = 0 \ \forall  Y\in V_1  \} 
\end{eqnarray*}
is a vector subspace. One can easily check that $\exp(V_2)\subset Z^\lambda$, hence $V_2\subset \mathfrak z^\lambda$. In particular,  $\mathfrak z^\lambda$ is an ideal and $Z^\lambda = \exp(\mathfrak z^\lambda)$ is a normal subgroup of $G$. Actually, one has $\mathfrak z^\lambda = (\mathfrak z^\lambda\cap V_1)\oplus V_2$.
The space $\mathfrak z^\lambda\cap V_1$ is by definition 
the kernel of the  skew-symmetric form on $V_1$, which we already encountered in \eqref{curv_map}, defined by
\begin{eqnarray*}
w(\lambda):(X,Y)\mapsto \lambda_2([X,Y]).
\end{eqnarray*}

If now $\gamma$ is a horizontal curve contained in $Z^\lambda$ (and hence abnormal) with $\gamma(0)=0$, then $\gamma$ is contained in the subgroup $H^\lambda$ generated by  $\mathfrak z^\lambda \cap V_1$, i.e.,
\begin{equation}\label{H:lambda}
H^\lambda :=\exp((\mathfrak z^\lambda \cap V_1)\oplus [\mathfrak z^\lambda \cap V_1,\mathfrak z^\lambda \cap V_1]).
\end{equation}
This implies that
\begin{eqnarray*}
{\rm Abn}(e) \subseteq \bigcup_{\substack{\lambda \in \mathfrak g^* \setminus \{0\}\\\lambda_1=0}} H^\lambda.
\end{eqnarray*}
It is interesting to notice that also  the reverse inclusion holds: indeed, for any $\lambda\in \mathfrak g^* \setminus \{0\}$ with $\lambda_1=0$ and any point $p\in H^\lambda$, there exists an horizontal curve $\gamma$ from the origin to $p$ that is entirely contained in $H^\lambda$;  $\gamma$ is then contained in $Z^\lambda$ and hence it is abnormal by Proposition~\ref{polynomial}. We  deduce that 
\begin{equation}\label{Abn=H}
{\rm Abn}(e) = \bigcup_{\substack{\lambda \in \mathfrak g^* \setminus \{0\}\\\lambda_1=0}} H^\lambda.
\end{equation}

We are now ready to prove a key fact in the setting of two-step Carnot groups: every abnormal curve is not abnormal in some subgroup.
We first recall that a \emph{Carnot subgroup} in a Carnot group is a Lie subgroup generated by a subspace of the first layer.

\begin{lem}\label{existence:subgroup}
Let $G$ be a $2$-step Carnot group. For each abnormal   curve $\gamma$ in $G$,
 there exists a proper Carnot subgroup $G'$ of $G$ containing $\gamma$, in which $\gamma$ is a non-abnormal horizontal curve.
\end{lem}
\begin{proof}
Let $\gamma$ be an abnormal   curve in $G$.
Then there exists 
$\lambda \in \mathfrak g^* \setminus \{0\}$, with $ \lambda_1=0$, such that
 $\gamma\subset H^\lambda$,
where $H^\lambda$ is the subgroup defined in \eqref{H:lambda}. By construction $H^\lambda$ is a Carnot subgroup. Since $\lambda\neq 0$ then $H^\lambda$ is a proper subgroup (of step $\leq 2$).

% By previous Lemma, $\gamma$ is contained in a proper Carnot subgroup $G_1$.
  If $\gamma$ is again abnormal in $H^\lambda$, then we iterate this process. %$\gamma$ is contained in a proper Carnot subgroup $G_2$ ou $G_1$. 
Since dimension decreases, after finitely many steps one reaches a proper Carnot subgroup $G'$ in which $\gamma$ is not abnormal.  
\end{proof}

\subsection{Parametrizing abnormal varieties within free two-step Carnot groups}

Let $G$ be a free-nilpotent 2-step Carnot group. Let $m\leq r:=\mathrm{dim}(V_1)$. Fix a $m$-dimensional vector subspace $W'_m\subset V_1$. Denote by $G_m$ the subgroup generated by $W'_m$, and $X_m=GL(r,\R)\times G_m$, equipped with the left-invariant distribution given at the origin by 
$W_m:=\{0\}\oplus W'_m$. Observe that $GL(r,\R)$ acts on $G$ by graded automorphisms. Let 
\begin{eqnarray*}
\Phi_m:X_m\to G,\quad (g,h)\mapsto g(h).
\end{eqnarray*}
In a polarized group $(X,V)$, given a submanifold $Y\subset X$, 
 the
{\em endpoint map relative to} $Y$ is  ${\End}^Y: Y\times L^2([0,1],V)\to X$, $(y, u)\mapsto \gamma^{(y)}_u(1)$, 
where $\gamma^{(y)}_u$ satisfies \eqref{ODE} with $\gamma^{(y)}_u(0)=y$.
We say that a horizontal curve $\gamma$ with control $u$ is \emph{non-singular relative to $Y$} if the differential at $(\gamma(0),u)$ of the endpoint map relative to $Y$
is onto. 

\begin{lem}\label{4325526360}
Let $G$ be a free 2-step Carnot group. For every abnormal curve $\gamma$ in $G$, there exists an integer $m<r$ and a horizontal curve $\sigma$ in $X_m$ such that $\Phi_m(\sigma)=\gamma$, and $\sigma$ is non-singular relative to $\Phi_m^{-1}(e)$.
\end{lem}

\begin{proof}
Let $\gamma$ be an abnormal curve in $G$ starting at $e$, with control $u$. By Lemma~\ref{existence:subgroup}, $\gamma$ is contained in the Carnot subgroup $G'$ of $G$ generated by some subspace $V'_1\subset V_1$ and is not abnormal in $G'$. Let $m=\mathrm{dim}(V'_1)$. Then there exists $g\in GL(r,\R)$ such that $V'_1=g(W'_m)$, and thus $G'=g(G_m)$. Let $\sigma=(g,g^{-1}(\gamma))$. This is a horizontal curve in $X_m$. Consider the endpoint map on $X_m$ relative to
%with starting point freely moving in 
the submanifold $\Phi_m^{-1}(e)=GL(r,\R)\times \{e\}$. Since $\gamma$ is not abnormal in $G'$, the image $I$ of the differential at $((g,e),g^{-1}(u))$ of the endpoint map contains $\{0\}\oplus T_{g^{-1}(\gamma(1))}G_m$. Every curve of the form $t\mapsto(k,g^{-1}(\gamma(t)))$ with fixed $k\in GL(r,\R)$ is horizontal, so $I$ contains $T_{g}(GL(r,\R))\oplus\{0\}$. One concludes that $I=T_{(g,\gamma(1))}X_m$, i.e., $\sigma$ is non-singular relative to  $\Phi_m^{-1}(e)$. By construction, $\Phi_m(\sigma)=\gamma$.
\end{proof}

\subsection{Application to general 2-step Carnot groups}

\begin{proposition}\label{general2step}
Let $G$ be a 2-step Carnot group. There exists a proper algebraic set $\Sigma\subset G$ that contains all abnormal curves leaving from the origin.
\end{proposition}

\begin{proof}
Let $f:\tilde{G}\to G$ be a surjective homomorphism from a free 2-step Carnot group of the same rank as $G$. Let $\gamma$ be an abnormal curve leaving from the origin in $G$. It has a (unique) horizontal lift $\tilde{\gamma}$ in $\tilde{G}$ leaving from the origin. According to Lemma~\ref{4325526360}, there exists an integer $m$ and a non-singular (relative to $\Phi_m^{-1}(e)$) horizontal curve $\sigma$ in $X_m$ such that $\Phi_m(\sigma)=\tilde{\gamma}$, i.e., $f\circ\Phi_m(\sigma)=\gamma$.
Namely, there exists $g\in GL(m, \R)$ such that $\sigma(t)=(g, g^{-1} \tilde \gamma(t))$.
 Consider the endpoint map ${\End}^Y$ on $X_m$ 
%with starting point freely moving in
relative to  the submanifold $Y:=\Phi_m^{-1}(e)$. 
Let us explain informally the idea of the conclusion of the proof. The composition 
 $f\circ\Phi_m \circ {\End}^Y$ is an  endpoint map for $G$,
  with starting point at the identity $e$.
  Hence, since the differential of ${\End}^Y$ at the control of $ {\sigma} $
   is onto,
    but 
     the differential of
     $f\circ\Phi_m\circ  {\End}^Y$ is not,
     the point  $\gamma(1)$ is a singular value of $f\circ\Phi_m$.
     Hence, we will conclude using  Sard's theorem.
     
     Let us now give a more formal proof of the last claims.
     Consider the map $\phi_m: Y\times L^2([0,1], W_m) \to L^2([0,1], V_1)$, defined as $(\phi_m(g,u))(t):= g(u(t)) \in V_1\subseteq T_e\tilde G$, for $t\in [0,1]$. We then point out the equality
     \begin{equation}\label{eq:Berna}
     f\circ\Phi_m\circ  {\End}^Y = {\End} \circ f_*\circ \psi_m,
     \end{equation}
where $\End:  L^2([0,1], V_1)\to G$ is the endpoint map of $G$ and $f_*:
 L^2([0,1], V_1) \to L^2([0,1], V_1) $
is the map
\[ (f_* (u))(t) = (\dd f)_e (u(t))\in V_1    \subseteq T_e  G.\]
Since $\sigma $ is   abnormal, i.e., the differential $\dd \End_{u_\gamma}$ is not surjective, and the differential of $\End^Y$ at the point
$(g,u_\sigma)=(f_* \circ \psi_m)u_\gamma$ is surjective,
from \eqref{eq:Berna}
we deduce that $\gamma(1)=\End^Y(g,u_\sigma)$ is a singular value for $f\circ\Phi_m$.
By the classical Sard Theorem, 
      the set $\Sigma_m$ of singular values of $f\circ\Phi_m$ has measure $0$ in $G$. So has the union
$\tilde \Sigma :=\cup_{m=1}^{r-1} \Sigma_m$
 of these sets. 
 By 
Tarski-Seidenberg's theorem 
%(the image of a semi-algebraic set by a semi-algebraic map is semi-algebraic)
 \cite[Proposition 2.2.7]{BCR},  
$\tilde \Sigma$ is a semi-algebraic set,
since the map $f\circ\Phi_m$ is algebraic and the set of critical points of an algebraic map is an algebraic set.
 Moreover, from \cite[Proposition 2.8.2]{BCR} we have  
 that
%$\dim(A)$ equals the  dimension of the  Euclidean closure of $A$   and  the  dimension of the Zariski closure of $A$, where 
%the last equality means that 
this semi-algebraic set is contained in an algebraic set $ \Sigma$ of the same dimension.
 Since $\tilde \Sigma$ has measure zero, the set $ \Sigma$ is  a proper algebraic set.
\end{proof}

\begin{example} [Abnormal curves not lying in any proper subgroup] \label{rmk33}
Key to our proof was the property,  encoded in Equation \eqref{H:lambda}, that every abnormal curve is contained in
   a proper subgroup of $G$. This property typically  fails for Carnot groups of step greater than $2$.
   Gol\'e and Karidi \cite{Gole_Karidi} 
    constructed a Carnot group of step 4 and rank 2    for which this property fails: namely, there is an
       abnormal curve that is   not contained in 
    any proper subgroup of their group. 
    Further on in this paper  (Section~\ref{not lying in any subgroup}) we show that this property fails   for the    free 3-step  rank-3 Carnot group.
\end{example}

\subsection{Codimension bounds on free 2-step Carnot groups} %: {proof of Theorem~\ref{thm:free_step_two}} 

In this section we prove Theorem~\ref{thm:free_step_two}; we will make extensive use of the result and notation of Section~\ref{facts2step}. In the sequel, we denote by $G$ a fixed free Carnot group of step 2 and by $r=\dim V_1$ its rank.

We  identify $G$ with its Lie algebra,  which   has the form $V \oplus \Lambda^2 V$
for $V=V_1 \cong \R^r$ a real vector space of dimension $r$. 
The Lie bracket is $[(v, \xi), (w, \eta) ] = (0, v \wedge w)$.
When we use the exponential map to identify the group with its Lie algebra,
the equation for a curve $(x(t), \xi(t))$ to be horizontal reads
\[
\dot x = u,\qquad\dot \xi = x \wedge u.
\]
If $W \subset V$ is a subspace, then the group it generates has the form 
$W \oplus \Lambda^2 W \subset V \oplus \Lambda^2 V$.

%Moreover,
%if $\gamma$ is an abnormal curve with $\gamma(0)=e$ and % then there exists 
%$\lambda \in \g^* \setminus \{0\}$ is such that
%$\gamma(t)\in Z^\lambda$, for all $t$, then the differential at $\gamma$ of the endpoint map 
%satisfies \footnote{Please, check!}
%\begin{equation}\label{containments}
%V_1 \subseteq{\rm Im}( {\rm d} \,{\rm End})_\gamma\subseteq (R_{{\rm End}(\gamma)})_*{\rm Ker}(\lambda)\subseteq T_{{\rm End}(\gamma)}M .
%\end{equation}
%where for all $\lambda \in \g^* \setminus \{0\}$ we set
%$$
%Z^\lambda :=\{ g\in G \;:\;  \left.( \Ad_g  )^*\lambda\right|_{V_1} \equiv 0\}
%=\exp\{ X\in \g \;:\;  \left.( e^{\ad_X}  )^*\lambda\right|_{V_1} \equiv 0\}.
%$$
%Recall that, for $X, Y\in \g$, the condition 
%$$(e^{\ad_X}  )^*\lambda (Y)=0$$
%means 
%$$ \lambda(Y+[X,Y] +\dfrac{1}{2}[X,[X,Y]]+\ldots) = 0.$$
%Notice that $\exp(V_s) \subseteq Z^\lambda $, if $s$ is the step of the group.

\subsection{Proof that ${\rm Abn}(e) $ is contained in a set of codimension $\geq 3$} % of   Theorem~\ref{thm:free_step_two}]
\label{sec:contained:codimension}
We use the view point discussed in Section~\ref{facts2step} where we defined the sets  $\mathfrak z^\lambda$ and $H^\lambda$.
We first claim that
\begin{equation}\label{rankr-2}
\text{dim}\: \mathfrak z^\lambda \cap V = \text{dim}\: \{X\in V \;:\;   \lambda_2([X,Y])  = 0 \ \forall\:  Y\in V   \}\leq r-2,
\end{equation}
for any $\lambda \in \mathfrak g^* \setminus \{0\}$ such that $\lambda_1=0$. Indeed, since $\lambda_2\not=0$, the alternating 2-form $w(\lambda):(X,Y)\mapsto\lambda_2([X,Y])$ has rank at least 2.
%Assume by contradiction that there exists such a $\lambda$ for which  dim $\mathfrak z^\lambda \cap V_1\geq r-1$; then, it can be easily seen that, since $G$ is free, the ideal generated by $\mathfrak z^\lambda \cap V_1$ contains all $V_2$.\footnote{If you think we should explain this point better, one could add the following sentence: ``Indeed, we have either $\mathfrak z^\lambda= V_1$ or $\mathfrak z^\lambda \cap V_1$ has a 1-dimensional complement in $V_1$, and in both cases the ideal generated by $\mathfrak z^\lambda \cap V_1$ contains all $V_2$.''} In particular, $\lambda_2$ annihilates all of $V_2$, i.e., $\lambda_2=0$; this is a contradiction because $\lambda_1=0$ and $\lambda\neq0$.

Then, by \eqref{rankr-2}, each $\mathfrak z^\lambda \cap V$ is contained in some $W\subset V$ with $\dim(W) = r-2$, hence $H^\lambda  \subseteq W \oplus \Lambda^2 W$ and, by \eqref{Abn=H},
\[
{\rm Abn}(e) = \bigcup_{\substack{\lambda \in \mathfrak g^* \setminus \{0\}\\\lambda_1=0}} H^\lambda \subseteq \bigcup_{W\in\,Gr(r,r-2)} W \oplus \Lambda^2 W.
\]
In fact, the equality 
\begin{equation}\label{formulaAbne}
{\rm Abn}(e) =  \bigcup_{W\in\,Gr(r,r-2)} W \oplus \Lambda^2 W.
\end{equation}
holds: this is because every codimension 2 subspace $W\subset V$ is the kernel of a rank 2 skew-symmetric 2-form (the pull-back of a nonzero form on the 2-dimensional space $V/W$), and every such skew-symmetric form corresponds to a covector $\lambda_2\in V_2^*=\Lambda^2 V^*$.

We now notice that the Grassmannian $Gr(r,r-2)$ of $(r-2)$-dimensional planes in $V$ has dimension $2(r-2)$ and that each $W \oplus \Lambda^2 W$ is (isomorphic to) the free group $\mathbb F_{m,2}$ of rank $m=r-2$ and step $2$, i.e., 
\[
\dim( W \oplus \Lambda^2 W) = m+\dfrac{m(m-1)}{2}=\frac{(r-1)(r-2)}{2}.
\]
It follows that the set $\cup_{W\in\,Gr(r,r-2)} W \oplus \Lambda^2 W$ can be parametrized with a number of parameters not greater than 
\[
\dim \mathbb F_{m,2}+ 
\dim Gr(r,m)
=\dfrac{r(r+1)}{2}  - 3.
\]
Since dim $G=r(r+1)/2$, 
the codimension $3$ stated in
Theorem~\ref{thm:free_step_two} now follows from \eqref{formulaAbne}. 
\qed
\subsection{Proof that ${\rm Abn}(e) $ is a semialgebraic  set of codimension $\geq 3$}
\label{sec:contained:codimension2}
Let $k=\lfloor (r-2)/2\rfloor$ and let $W$ be a codimension 2 vector subspace of $V_1$. Every pair $(\xi,\eta)\in W\oplus\Lambda^2 W$ can be written as
\begin{eqnarray*}
\xi=\sum_{j=1}^{r-2}x_j \xi_j,\quad
\eta=\sum_{i=1}^{k}z_i\xi_{2i-1}\wedge\xi_{2i},
\end{eqnarray*}
for some $(r-2)$-uple of vectors (e.g., a basis) $(\xi_j)_{1\leq j\leq r-2}$ of $W$. Conversely, every pair $(\xi,\eta)\in\mathfrak{g}=V\oplus\Lambda^2V$ of this form belongs to $W\oplus\Lambda^2 W$ for some codimension 2 subspace $W$ of $V_1$. Therefore  
\begin{eqnarray*}
\bigcup_{W\in\,Gr(r,r-2)} W \oplus \Lambda^2 W
\end{eqnarray*}
is the projection on the first factor of the algebraic subset
\begin{eqnarray*}
\{(\xi,\eta,\xi_1,\ldots,\xi_{r-2},x_1,\ldots,x_{r-2},z_1,\ldots,z_{k})\,:\,\xi=\sum_{j=1}^{r-2}x_j \xi_j,\,
\eta=\sum_{i=1}^{k}z_i\xi_{2i-1}\wedge\xi_{2i}\}
\end{eqnarray*}
of $\mathfrak{g}\times V^{r-2}\times\R^{r-2}\times\R^{k}$. Since the exponential map is an algebraic isomorphism, $\Abn(e)=\bigcup_{W\in\,Gr(r,r-2)} W \oplus \Lambda^2 W$ is semi-algebraic, and it is contained in an algebraic set of the same codimension (see
\cite[Proposition 2.8.2]{BCR}).
\qed

%\begin{remark}(Precise description of $\Abn(e)$)\footnote{Richard: is this what you had in mind?}
%When $G$ is a free Carnot group of step 2 and rank $r$, one can improve formula \eqref{formulaAbne} to prove that %actually
%\[
%{\rm Abn}(e) =  \bigcup_{W\in\,Gr(r,r-2)} \exp (W \oplus [W, W]).
%\]
%It is enough to show that, for any fixed $(r-2)$-dimensional plane $W\subset V_1$, the set $\exp(W \oplus [W, W])$ %is contained in ${\rm Abn}(e)$. Since $G$ is free, we have
%\[
%\text{dim }[W,V_1]=\frac{r(r-1)}{2}-1=\text{dim }V_2-1,
%\]
%hence there exists $Y\in V_2\setminus\{0\}$ such that $Y\notin [W,V_1]$. We can then define $\lambda\in \mathfrak %g^* \setminus \{0\}$ with $\lambda_1=0$ by
%\[
%\lambda_{|V_1\oplus[W,V_1]}=0,\quad \lambda(Y)=1.
%\]
%One can easily check that $\mathfrak z^\lambda \cap V_1=W$ and, in particular, $\exp(W \oplus [W, W])=H^\lambda$. %We then conclude by \eqref{Abn=H}.
%\end{remark}

In the rest of this section we proceed with the  more precise description of the set ${\rm Abn}(e)$, as described in   Theorem~\ref{thm:free_step_two}.

Each $\xi \in \Lambda^2 V$ can be viewed, by contraction, as
a linear skew symmetric  map $\xi: V^* \to V$. For example, if $\xi = v \wedge w$, then
this map sends $\alpha \in V^*$ to $\alpha(v) w-\alpha(w) v$.

\begin{definition} For $\xi \in \Lambda^2 V$ let ${\rm supp}(\xi) \subset V$ denote the image of $\xi$,
when $\xi$ is viewed as a linear map $V^* \to V$. For $(v, \xi) \in V \oplus \Lambda^2 V$ set
${\rm supp}(v, \xi) = \R v + {\rm supp}(\xi).$ 
Finally, set ${\rm rank}(v, \xi) = {\rm dim}({\rm supp}(v, \xi))$.
\end{definition} 

\begin{proposition} \label{prop:RM}
 If  $G$ is the free 2-step nilpotent group on $r$ generators then  \[\Abn(e) = \{ (v, \xi) :  {\rm rank}(v, \xi) \le r-2 \}.\]
\end{proposition}

\begin{proof}
%Let us recall \eqref{formulaAbne}:
%\[
%{\rm Abn}(e) = \bigcup_{W\in\,Gr(r,r-2)}  W \oplus \Lambda^2 W.
%\]
From  \eqref{formulaAbne} we can directly  derive the new characterization. Suppose that $W \subset V$ is any subspace  and     $(w, \xi) \in W \oplus \Lambda^2 W $.
Then clearly ${\rm supp}(w,\xi) \subset W$.   
  Conversely, if $(w, \xi)$ has support a subspace of   $W$, then 
 one easily checks that $(w, \xi) \in W \oplus \Lambda ^2 W$. Taking $W$ an arbitrary subspace of rank $r-2$
 the result follows.   
\end{proof}

 By combining  Proposition~\ref{prop:RM} with some linear algebra we will conclude the proof  of   Theorem~\ref{thm:free_step_two}.
This proof is independent of   Sections~\ref{sec:contained:codimension}
and~\ref{sec:contained:codimension2} and yields a different perspective on the abnormal set.

%Using Proposition~\ref{prop:RM} we conclude the proof  of   Theorem~\ref{thm:free_step_two}. This   proof
%%, though slightly longer then the previous one, 
%allows a  more precise description of the set $\Abn(e)$ as an algebraic set.
%On purpose, we decided to give an argument that does not rely on Sections~\ref{sec:contained:codimension}
%and~\ref{sec:contained:codimension2} to give a different view point using linear algebra.}

%\begin{corollary} If $G$ is the free 2-step nilpotent group on $r$ generators, then
%$\Abn(e)$ is an affine algebraic variety of codimension $3$.
%\end{corollary}

\begin{proof}[Proof of     Theorem~\ref{thm:free_step_two}]
Let $G$ be the free-nilpotent 2-step  group on $r$ generators.
First, we write the polynomials defining $\Abn(e)$, then we compute dimensions.
It is simpler to divide up into the case of even and odd rank $r$.
We will   consider the case of even rank   in detail and leave most of the
odd rank case up to the reader.   

The linear algebraic Darboux theorem  will prove useful for computations.
All bivectors have even rank. 
This theorem asserts that the   bivector $\xi \in \Lambda ^2 V$ has rank $2m$   if and only
if there exists $2m$ linearly independent vectors  $e_1, f_1, e_2 , f_2 , \ldots e_m, f_m$ in  $V$
such that $\xi = \Sigma_{i=1} ^m e_i \wedge f_i$.

Let us now specialize to the case where $r = {\rm dim}(V)$ is even.
Write
\[r =2s.\]
Using Darboux one checks that ${\rm rank}(0, \xi) \le r-2$
if and only if   $\xi^s = 0$ (written out in components,
$\xi$ is a skew-symmetric $2r \times 2r$ matrix and the vanishing
of $\xi^s$ is exactly the vanishing of the Pfaffian of this matrix). Now, if ${\rm rank}(0,\xi) = r-2$
and  ${\rm rank}(v, \xi) \le r-2$, it must be the case that $v \in {\rm supp}(\xi)$; equivalently, in the Darboux basis, $v = \Sigma_{i=1} ^m a_i e_i + \Sigma_{i=1} ^m b_i f_i $. 
It follows  in this case that $v \in {\rm supp}(\xi)$ if and only if   
$v \wedge \xi^{s-1} = 0$.   Now, if ${\rm rank}(0, \xi) < r-2$ then ${\rm rank}(0, \xi) \le r-4$
and so ${\rm rank}(v, \xi) \le r-3$ for any $v \in V$.  But  ${\rm rank}(0, \xi) < r-2$ if and only
if $\xi^{s-1} = 0$ in which case automatically $v \wedge \xi^{s-1} = 0$.

We have proven that in the case $r =2s$, the equations for $\Abn(e)$ 
%in the case of even rank 
are the polynomial equations $\xi^s =0$ and $v \wedge \xi ^{s-1} = 0$.

To compute dimension, we   stratify $\Abn(e)$ according to the rank of its elements.    
The dimensions of the strata are easily checked to decrease with decreasing rank, so that the dimension of $\Abn(e)$
equals the dimension of the largest stratum, the stratum consisting of  the  $(v, \xi)$ of even rank $r-2$. 
(The  Darboux theorem and a bit of work yields that   the  stratum  having rank $k$ with $k$ odd  consists of exactly one $Gl(V)$ orbit 
while the stratum having rank $k$ with $k$ even consists of   exactly  two  $Gl(V)$ orbits). 
A point $(v, \xi)$ is in  this stratum if and only if  $\xi^s = 0$ while  $\xi ^{s-1} \ne 0$ and $v \in {\rm supp}(\xi)$.
Let us put the condition on $v$ aside for the moment.  The first condition on $\xi$ is the 
Pfaffian equation which defines an algebraic hypersurface in $\Lambda^2 V$, the zero locus of the Pfaffian
of $\xi$.  The second equation for $\xi$ defines the smooth locus of the Pfaffian. Thus,
the set of $\xi$'s satisfying the first two equations has dimension $1$ less than that of $\Lambda^2 V$,
so its dimension is ${r \choose 2} -1$.
Now, on this smooth locus $\{Pf = 0 \}_{\rm smooth} \subset \{ Pf = 0 \}$ we have a well-defined algebraic map  $F: \{Pf = 0 \}_{\rm smooth} \to Gr(r, r-2)$
which sends $\xi$ to $F(\xi) = {\rm supp}(\xi)$.    Let  $U \to Gr(r, r-2)$  denote the canonical  rank $r-2$
vector bundle over the Grassmannian.  Thus $U \subset \R^r \times Gr(r, r-2)$ consists of pairs $(v,P)$ such that $v \in P$. 
 Then  $F^* U$ is a rank $r-2$ vector bundle over $\{Pf = 0 \}_{\rm smooth}$ consisting of pairs $(v, \xi) \in \R^2 \times \Lambda^2 V$
such that $v \in {\rm supp}(\xi)$ and $\xi$ has rank $r-2$.  In other words, the additional  condition $v \in {\rm supp}(\xi)$
says exactly that $(v, \xi) \in F^*U$.  It follows that the   dimension of this principle stratum is ${\rm dim}(F^* U) = ({r \choose 2} -1) + (r-2) = {\rm dim}(G) -3$.

Regarding the odd rank case 
\[r = 2 s+ 1\]
the same logic shows that the equations defining $\Abn(e)$
are $\xi^s = 0$ and involves no condition on $v$.  A well-known matrix computation \cite{Arnold_1971}
shows that the subvariety  $\{ \xi ^s = 0 \}$ in the odd rank case has codimension $3$. 
%(Think about the case $s =1$.) 
Since the map $V \oplus \Lambda ^2 V \to \Lambda ^2 V$
is a projection, and since $\Abn(e)$ is the inverse image of $\{ \xi ^s = 0 \} \subset \Lambda ^2 V$
under this projection, its image remains codimension $3$.
\end{proof}

Recall that the rank of $\xi\in \Lambda^2 V$ is the (even) dimension $d$ of its support.   For an open dense subset of elements of $  \Lambda ^2 V$,  the rank  is as large as possible: $r$ if $r$ is even and   $r-1$ if $r$ is odd.  We call {\em singular} the elements $\xi\in\Lambda ^2 V$ whose rank is less than the maximum and we write $(\Lambda^2 V)_{\rm sing}$ to denote the set of singular elements. From Proposition~\ref{prop:RM} we easily deduce the following.

\begin{proposition}
 The projection of  $\Abn(e)$   onto $\Lambda^2 V$
 coincides with  the singular elements  
$ (\Lambda^2 V)_{\rm sing} \subset \Lambda^2 V$. 
\end{proposition}

\begin{remark}
A consequence of the previous result is the fact that elements of the form $(0, \xi)$ where ${\rm rank}(\xi)$ is maximal  can never be reached by abnormal curves. Notice that such elements are in the center of the group.  %This suggests   finding ``singularly untouchable  '' elements in the center of general Carnot groups by  appropriately quantifying ``generic''.   
\end{remark}
 
To be more  precise about   $\Abn(e)$ we must divide into two cases according to the parity of $r$.
%This division into cases is necessary for several reasons: for instance, notice that   the codimension of $\Lambda^2 V _{sing}$ is $1$
%if $r$ is even and $3$ if $r$ is odd. %(See par.  following next theorem for the computation.)

\begin{theorem} 
If $G=V\oplus\Lambda^2 V$ is a free Carnot group with odd rank $r$, then   $\Abn(e) = V \oplus (\Lambda ^2 V)_{\rm sing}$.
\end{theorem}

The previous result, as well as the following one, easily follows from Proposition~\ref{prop:RM}. To describe the situation for $r$ even, let us  write $(\Lambda ^2 V)_d$ for those elements of $\Lambda ^2 V$ whose rank is exactly $d$
and $(\Lambda ^2 V)_{<d}$ for those elements whose rank is strictly less than $d$.  

%In Arnol'd-Gusein-Zade (and many other places) \footnote{A precise reference is needed here. I (Davide) was not able to find one.} one can find a computation for the codimension of $(\Lambda ^2 V)_{<d +1}$ which is  also that of $(\Lambda^2 V)_d$. Let us review how this goes. 
%
%Choose $\xi_0  \in (\Lambda^2 V)_d$.  By Darboux theorem, we can choose a basis $u_1, v_1, u_2, v_2, \ldots, u_{\ell} , v_{\ell}, f_1, \ldots f_t$ for $V$, where $2 \ell = d$
%such that   $\xi_0 = \sum_i u_i \wedge v_i$.   Take another nearby element $\xi = \xi_0 + h$. 
%Now, in our basis, any $\xi$ is expressible as a skew-symmetric matrix.
%The matrix of $\xi_0$ has a block form with nonzero two-by-twos on the diagonal in this basis.
%A linear analysis shows that the diagonal $t \times t$ matrix left over -that concerning  the $f_i$'s -- must be identically zero if $\xi = \xi_0 + h$ is to have rank $d$ to first order.
%Thus the codimension of $(\Lambda^2 V)_d$ is the dimension of the space $\Lambda ^2 (\R^t)$ where $t =r-d$ is the corank of $\xi_0$.
%
%In particular, for $r$ even, $\Lambda ^2 V_{sing} = (\Lambda^2 V) _{r-2}$ has codimension 1 and so is a hypersurface.  The polynomial whose vanishing defines this hypersurface
%is a famous polynomial, the Pfaffian. It can be written $Pf(\xi) = \xi^m/ (e_1 \wedge e_2 \wedge \ldots \wedge e_r)$ where $2m = r$. 
%
%We are now ready to treat the even-dimensional case.

\begin{theorem} 
If $G=V\oplus\Lambda^2 V$ is a free Carnot group with even rank $r$, then $\Abn(e)$   is the union $Y\cup Y_1$ of the two quasiprojective subvarieties
\[
\begin{split}
& Y= \{ (v, \xi ) \in V \oplus \Lambda^2 V  : v \in {\rm supp}(\xi),  \xi \in (\Lambda ^2 V)_{r-2} \}\\
& Y_1 = V \times (\Lambda^2 V)_{< r-2}.
\end{split}
\]
In particular, $\Abn(e)$ is a singular  algebraic variety of codimension $3$.
\end{theorem}

%\begin{remark}
We observe  that $Y_1 = \bar Y \setminus Y$.  

\begin{remark}
Given any $g = (v, \xi) \in G$ we can define its {\em singular rank} 
to be the minimum of the dimensions of the image of the differential of the endpoint map  $\dd\End(\gamma)$,
where the minimum is taken over all $\gamma$ that connect $0$ to $g$. 
Thus, the singular rank of  $g= 0$ is $r$ and is  realized by the constant curve, while if $\xi$ is generic then the singular rank of $g= (0,\xi)$ is ${\rm dim}(G)$,  which means that  every horizontal curve connecting $0$ to $g$ is not abnormal.
%Thus, the singular rank of  $g= 0$ is $r$, realized by the constant curve, while if $\xi$ is generic then the singular rank is ${\rm %dim}(G) = {{r+1} \choose 2}$ the curve is not abnormal.\footnote{what does this sentence mean?}

It can be easily proved that, if $r$ is even and  $v \in {\rm supp}(\xi)$, then the singular rank of $g$ is
just ${\rm rank}(\xi)$. In this case we take a $\lambda$ with ${\rm ker}(\lambda) = {\rm supp}(\xi)$ and realize
$g$ by any horizontal curve lying inside $G(\lambda)$.
\end{remark}

\section{Sufficient condition for Sard's property}   	\label{sec:sufficient}
In Section~\ref{EndpointMap} we observed that, given a polarized group $(G,V)$ and a horizontal curve $\gamma$ such that $\gamma(0)=e$ and with control $u$, the space
$
(\dd R_{\gamma(1)})_eV+(\dd L_{\gamma(1)})_e V +\mathcal S(\gamma(1))
$
is a subset of $ {\rm Im}(\dd \End_u)\subset T_{\gamma(1)}G$. Therefore, if $g\in G$ is such that
\begin{equation}\label{Eq_on_e}
\Ad_{g^{-1}} V + V + (\dd L_g)^{-1}\mathcal X(g) = \g,
\end{equation}
for some subset $\mathcal X$ of $\mathcal S$, then $g$ is not a singular value of the endpoint map. Here we denoted with $\mathcal X(g)$ the space of vector fields in $\mathcal X$ evaluated at $g$.
 In particular, if the equation above is of polynomial type (resp.~analytic), then $(G,V)$ has the Algebraic (resp.~Analytic) Sard Property.

In the following we embed both sides of \eqref{Eq_on_e} in a larger Lie algebra $\tilde{\g}$, and we find conditions on $\tilde{\g}$ 
that are sufficient for \eqref{Eq_on_e} to hold. 
The idea is to consider a group $\tilde G$ that acts, locally, on $G$ via contact mappings, that is, diffeomorphisms that preserve the left-invariant subbundle $\Delta$. It turns out that the Lie algebra $\tilde{\g}$ of $\tilde G$, viewed as algebra of  left-invariant vector fields on $\tilde G$, represents a space of contact vector fields of $G$.

%\section{Polynomials via a prolongation} 
\subsection{Algebraic prolongation}\label{algebraic_prolongation}
Let  $\tilde G$ be a Lie group and  $G$ and $H$ two  subgroups.
Denote by $\tilde \g$, $\g$, and  $\h$ the respective Lie algebras seen as tangent spaces at the identity elements.
We shall assume that $H$ is closed.
Suppose that 
$\tilde \g = \h \oplus \g$ and that 
we are given the decompositions in vector space direct sum 
\[\mathfrak h  = V_{-h}\oplus \cdots\oplus V_{0} \]
and
\[\g= V_{1}\oplus \cdots\oplus V_{s} \]
% and note that it is a subalgebra of $\g$
%Assume that 
in such a way that $\tilde \g$ is graded, 
%is graded
%admits a  grading
% $$\tilde \g:= %{\rm Prol}(\g):=
% V_{-k}\oplus \cdots\oplus V_{s}, $$
% for some $k\geq0$, 
 namely  $[V_i, V_j]\subseteq V_{i+j}$, for $i,j=-h,\ldots, s$,
and   
%$\g$  admits a statification
%$$\g= V_{1}\oplus \cdots\oplus V_{s} $$
  $\g$ is stratified, i.e., $[V_1, V_j]=V_{j+1}$ for $j>0$.
In other words, 
$\tilde \g$ is  a (finite-dimensional)
prolongation of the Carnot %stratified
 algebra $\g$.

% Let $\tilde G$ be the simply connected Lie group with Lie algebra $\tilde \g$. (why does it exist?)
 
%The algebra $\g$ can be seen as a subalgebra of the  Lie algebra $\tilde \g$ and  the group $G$   is a   subgroup of $\tilde G$.
%Moreover, 
%Let $\mathfrak h : = V_{-h}\oplus \cdots\oplus V_{0} $ and note that it is a subalgebra of $\g$.
% Let $H$ be the closed\footnote{Should we assume that $H$ is closed or is it alway the case?} subgroup of $\tilde G$ with Lie algebra $\mathfrak h$. 

%We give the warning that it may not be the case that  
%$\tilde G=  GH$ nor that $G\cap H=\emptyset$,  moreover, neither of the subgroups may be normal. 
%Nonetheless w

We have a local embedding of
$G$ within the quotient space
$\tilde G/H:= \{gH: g\in G\}$ via 
%the map
% \begin{eqnarray*}
% G  &\to &\tilde G/H\\
% g&\mapsto& gH,
% \end{eqnarray*}
% which is 
 the restriction to $G$ of the projection
  \begin{eqnarray*}
 \pi:\tilde G  &\to &\tilde G/H\\
 p&\mapsto& \pi(p):=[p]:=pH.
 \end{eqnarray*}
The group 
$\tilde G$ acts on $\tilde G/H$ on the left:
\begin{eqnarray*}
 \bar L_{\tilde g} :\tilde G/H &\to &\tilde G/H\\
gH&\mapsto& \bar L_{\tilde g}(gH):= \tilde g g H.
\end{eqnarray*}
We will repeatedly use     the  identity
\begin{equation}
\label{Lpi_piL}
\bar L_{\tilde g} \circ\pi   = \pi  \circ L_{{\tilde g} }  .    \end{equation}
%which is true since, for all $p\in \tilde G$, we have
%      $\bar L_{\tilde g} \circ\pi (p)  = 
%     {\tilde g} p H  = 
%     \pi ( {\tilde g} p) = 
%     \pi  \circ L_{{\tilde g} } (p)  $.

On the groups $\tilde G$ and $G$ we consider the two left-invariant subbundles $\tilde\Delta$ and $\Delta$
 that, respectively, are defined by
 \begin{eqnarray*}	
\tilde \Delta_e&:=& \mathfrak h + V_1,\\
 \Delta_e&:=&  V_1.
 \end{eqnarray*}	
 Notice that both subbundles are bracket generating $\tilde\g$ and $\g$, respectively.
 Moreover, $\tilde \Delta$ is $\ad_{\mathfrak h}$-invariant, hence it passes to the quotient as a $\tilde G$-invariant subbundle $\bar \Delta$ on 
 $\tilde G/H $.
 Namely, there exists a subbundle $\bar \Delta$ of the tangent bundle of 
 $\tilde G/H $ such that
 \[\bar \Delta= {\rm d} \pi (\tilde \Delta).\]

 \begin{lemma}\label{contacto:morphism}
 The map
  \begin{eqnarray*}
 i:=\pi_{|_G}  : (G,\Delta)  &\to &(\tilde G/H, \bar \Delta)\\
 g&\mapsto& gH
 \end{eqnarray*}
 is a local diffeomorphism and preserves the subbundles, i.e., it is locally a contacto-morphism. 
 \end{lemma}
 \begin{proof}
 Since $\g$ is a complementary subspace of $\mathfrak h$ in $\tilde \g$,  the differential
 $({\rm d} i)_e$ is an isomorphism between 
 $\g$ and 
 $T_{[e]} \tilde G/H$.
Since by Equation \eqref{Lpi_piL} the map $\pi$ is $G$-equivariant, then  
 $({\rm d} i)_g$ is an isomorphism for any arbitrary  $g\in G$. 
% Let $g\in G$ an arbitrary point. 
% We claim that  $({\rm d} i)_g$ is injective. Indeed,
%  $({\rm d} i)_g = $
Hence, the map $i$ is a local diffeomorphism.
If $X$ is a left-invariant section of   $ \Delta$ then 
 \[({\rm d} i)_g X_g = 
 \left. \dfrac{{\rm d}} {{\rm d} t} [g \exp(t X_e)]\right|_{t=0}\in \bar \Delta_{[g]},\]
since $X_e\in V_1$.
 \end{proof}

  Let $\pi_\g:\tilde \g = V_{-h}\oplus \cdots\oplus V_{0} \oplus \g \to \g$ be the projection induced by the direct sum.
%    Let $\pi_G:\tilde G   \to G\simeq \tilde G/H$ be the projection induced by quotienting.    
%    Here are few relations:
The projections $\pi$ and $\pi_\g$ are related by the following equation:
    \begin{equation}
    \label{diff:pi:e} 
({\rm d}\pi)_e = {({\rm d}\pi)_e}_{|_\g}\pi_\g .   \end{equation}
    Indeed, if $Y\in \g$, then the formula trivially holds;
%    $({\rm d}\pi_G)_e Y= \left.  \frac{\dd}{\dd t} \exp(tY) H \right|_{t=0}
%    \simeq 
%     \left.  \frac{\dd}{\dd t} \exp(tY)  \right|_{t=0}= Y$; 
if 
     $Y\in \mathfrak h$ , then 
    $({\rm d}\pi)_e Y= \left.  \frac{\dd}{\dd t} \exp(tY) H \right|_{t=0}
   =
     \left.  \frac{\dd}{\dd t}H \right|_{t=0}= 0$.

        The differential of the projection $\pi$ at an arbitrary point $\tilde g$ can be expressed using the projection $\pi_\g$ via 
     the following equation:
\begin{equation}
\label{diff:pi} 
    ({\rm d}\pi)_{\tilde g} = ({\rm d} (\bar L_{\tilde g} \circ \pi_{|_G}))_e\circ \pi_\g  \circ ( {\rm d} L_{{\tilde g}^{-1}})_{\tilde g} .    \end{equation}
      Indeed, first notice that 
      $  ({\rm d} \pi_{|_G})_e  = {({\rm d}\pi)_e}_{|_\g}$, then
       from \eqref{diff:pi:e} and \eqref{Lpi_piL}  we get
      \begin{eqnarray*}
      ({\rm d} (\bar L_{\tilde g} \circ \pi_{|_G}))_e\circ \pi_\g  \circ ( {\rm d} L_{{\tilde g}^{-1}})_{\tilde g} &=&
      ({\rm d} \bar L_{\tilde g})_{[e]} \circ {({\rm d}\pi)_e}_{|_\g}\circ \pi_\g  \circ ( {\rm d} L_{{\tilde g}^{-1}})_{\tilde g} \\&=&
 %       (\bar L_{\tilde g})_* \pi_\g (L_{{\tilde g}^{-1}})_* =
 ({\rm d} \bar L_{\tilde g})_{[e]}\circ ({\rm d}\pi)_e  \circ ( {\rm d} L_{{\tilde g}^{-1}})_{\tilde g} \\&=&
%        (\bar L_{\tilde g})_* ({\rm d}\pi_G)_e  (L_{{\tilde g}^{-1}})_* =
{\rm d} (\bar L_{\tilde g} \circ  \pi  \circ  (L_{{\tilde g}})^{-1})_{\tilde g}=
%          (L_g)_* ({\rm d}\pi_G)_e  (L_{g^{-1}})_* =
       ({\rm d}\pi)_{\tilde g}  .
             \end{eqnarray*} 
      
%      \newpage %$$ $$
\subsection{Induced contact vector fields}      
     To any vector $X\in T_e\tilde G \simeq \tilde \g$ we want to associate a contact vector field $X^G$ on $G$.
     Let $X^R$ be the right-invariant vector field on $\tilde G$ associated to $X$.
     We define $X^G$ as the (unique) vector field on $G$ with the property that 
    \begin{eqnarray*}
    {\rm d}\pi (X^R) ={\rm d}  i ( X^G),
     \end{eqnarray*}
     as vector fields on $i(G)$.
     In other words, we observe that there exists a (unique) vector field $\bar X$ on $\tilde G/H$ that is $\pi$-related to $X^R$ and $i$-related to some (unique) $X^G$.
The flow of $X^R$ consists of left translations in $\tilde G$, hence they pass to the quotient $ \tilde G/H$. Thus $\bar X$ shall be the vector field on $\tilde G/H$ whose flow is 
%The two interpretations are equivalent and the flow of $X^G$ is the map
\begin{eqnarray*}
\Phi^t_{\bar X} (gH) = \pi(\exp(tX)g )=\exp(tX)g H
=
\bar L_{\exp(tX)}(gH).
     \end{eqnarray*}
%  Interpretation 2:
 In other words, we define 
$\bar X$ as the    vector field on $\tilde G/H$ as %for which
% $$(X^G)_g  := (L_g)_* \pi_\g (\Ad_{\g^{-1}} X), \qquad \forall g\in G.$$
 \begin{equation}\label{Interpretation:2}
\bar X_{[p]} : = ({\rm d} \pi)   (X^R)_p 
%=\left. \frac{\dd}{\dd t}    \pi (\Phi^t_{X^R}  (p))\right|_{t=0}
=\left.  \frac{\dd}{\dd t}\pi(\exp(tX)p )\right|_{t=0}, \qquad \forall p\in \tilde  G.
\end{equation}
\begin{definition}
For all $X\in \tilde \g$ and $g\in G$, we set 
\begin{eqnarray*}
(X^G)_g:= 
({\rm d}( \pi_{|_G})_g)^{-1} ({\rm d} \pi)_g ({\rm d} R_g)_e X.
     \end{eqnarray*}
\end{definition}

%    
% Interpretation  1:
% 
%
% We define 
From \eqref{diff:pi}, the vector field $ X^G$ satisfies %has the properties that
%as the    vector field on $G$ for which
  \begin{equation}
%  \bar X_{[p]}  
  (X^G)_g
%   = ({\rm d} \pi)   (X^R)_p 
%     = ({\rm d}\pi)_p ({\rm d} R_p) X_e
%  = (\bar L_g)_* \pi_\g (L_{g^{-1}})_* (R_g)_* X_e
     =   {\rm d}({ L_g }_{|_G})_e \pi_\g \Ad_{g^{-1}} X
  , \quad \forall g\in   G,
  \end{equation}
%  and hence also that 
%    \begin{equation}\bar X_{[p]}  
%%   = ( \pi_G)_*   (X^R)_g 
%%     = ({\rm d}\pi_G)_g (R_g)_* X_e
%  = ({\rm d} \bar L_p) \pi_\g ({\rm d} L_{p^{-1}}) ({\rm d} R_p) X_e
%     = ( {\rm d} \bar L_p) \pi_\g (\Ad_{p^{-1}} X)
%  , \quad \forall p\in \tilde G.
%  \end{equation}

%  Interpretation 2:
%  
%
% We define 
%$X^G$ as the    vector field on $G$ for which
%% $$(X^G)_g  := (L_g)_* \pi_\g (\Ad_{\g^{-1}} X), \qquad \forall g\in G.$$
% \begin{equation}\label{Interpretation:2}
%(X^G)_g =\left. \frac{\dd}{\dd t}    \pi_G (\Phi^t_{X^R}  (g))\right|_{t=0}
%=\left.  \frac{\dd}{\dd t}\pi_G(\exp(tX)g )\right|_{t=0}, \qquad \forall g\in G.
%\end{equation}

%The two interpretations are equivalent and the flow of $X^G$ is the map
%$$\Phi^t_{X^G} (g) = \pi_G(\exp(tX)g )=\exp(tX)g H
%=
%\bar L_{\exp(tX)}(g).$$

We remark  that if $X\in \g \subset \tilde \g$ then 
 $X^G = X^R$, as vector fields in $G$.

\begin{proposition}\label{L:bar:contact}
Let $X^G$ be the 
 vector field defined above.
 Then
 
 i)
  $X^G$ has polynomial components when read in exponential coordinates.

ii)   $X^G$ is a contact vector field, i.e., its flow preserves $ \Delta$.

%(Equivalently, for all $Y\in V_1$ we have that $[\bar X, Y^L] \in \bar \Delta.$ )
\end{proposition}
\proof
Because the algebra $\tilde \g$ is graded, we have that for every $X\in \g$ the map $\ad_X$ is a nilpotent transformation of $\tilde \g$.
Consequently,
for all $g\in G$, the map $\Ad_g$ is a polynomial map of $\tilde \g$.
Therefore, in exponential coordinates, $X^R _{|_G}$ is a polynomial vector field and   $X^G$ is as well.

We next   show that 
the vector field  in \eqref{Interpretation:2} is contact, in tother words, each map $\bar L_{p}$ preserves $\bar \Delta$.
%Proof.
Any vector in $\bar \Delta$ is of the form 
${\rm d} \pi (Y^L_{\tilde g})$
with $Y_e \in
%\bar \Delta_{[e]}= {\rm d} \pi (
\mathfrak h +V_1$ and ${\tilde g}\in \tilde G$.
We want to show that 
$
({\rm d} \bar L_{ p})_{[\tilde g] }( {\rm d} \pi)_{\tilde g} (Y^L_{\tilde g})$ is in $\bar \Delta$. In fact,
using \eqref{Lpi_piL}, we have
%Let $Y_e\in V_1$ and $g\in G$. Write $\tilde g g $ as $g' h'$, for $g'\in G$ and $h' \in H$, hence $g'=\pi(\tilde g  g)=\bar L_{\tilde g}(g)$. 
%Moreover, we write $h'$ as $\exp(Z)$ with $Z\in \mathfrak h=V_{-h}\oplus \cdots\oplus V_{0} $.
%Then
\begin{eqnarray*}
({\rm d} \bar L_{ p})_{[\tilde g] }({\rm d} \pi)_{\tilde g} (Y^L_{\tilde g})
&=&
{\rm d} (\bar L_{ p}\circ  \pi)_{\tilde g} (Y^L_{\tilde g})\\
&=&
{\rm d} (\pi \circ L_{ p}  )_{\tilde g} (Y^L_{\tilde g})\\
&=&
{\rm d} \pi_{p\tilde g } ( {\rm d}  L_{ p}  )_{\tilde g} (Y^L_{\tilde g})\\
&=&
{\rm d} \pi_{p\tilde g }   (Y^L_{p\tilde g}) \in {\rm d} \pi (\tilde \Delta).
%((\bar L_{\tilde g})_* Y^L)_g
%&=&
%\left.  \frac{\dd}{\dd t}  \tilde g g \exp(tY) H\right|_{t=0}\\
%&=&
%\left.  \frac{\dd}{\dd t} g' h' \exp(tY) (h')^{-1}H\right|_{t=0}\\
%&=&
%(L_{g'})_* ({\rm d} \pi)_e (C_{h'})_* Y_e\\
%&=&
%(L_{g'})_*   \pi_\g \Ad_{h'}  Y_e\\
%&=&
%(L_{g'})_*   \pi_\g e^{\ad_Z}   Y_e\\
%&=&
%(L_{g'})_*   \pi_\g (   Y_e +[Z,Y_e] +\frac{1}{2}[Z,[Z,Y_e]] +...).
\end{eqnarray*}
%We shall prove in Proposition~\ref{L:bar:contact} that 
Now that we know that $\bar X$ is a contact vector field of $\tilde G/H$,
%. In other words, we will
%%To show this   we need to
% show that  the maps of the form $\bar L_{p}$, for $p\in \tilde G$, preserve the distribution $\bar \Delta$. %, 
% \begin{eqnarray*}
%\rho_{\tilde g}:G &\to& G\\
%g&\mapsto& \rho_{\tilde g}(g):=\bar L_{\tilde g}(g)=  \pi_G ({\tilde g} g)\simeq {\tilde g} g H.
% \end{eqnarray*}
%As a consequence, 
from Lemma~\ref{contacto:morphism} we deduce that 
 the vector field $X^G$, which satisfies $\bar X = {\rm d} i ( X^G)$, is 
 a contact vector field on $G$.
% , such that 
% $\bar X = {\rm d} i ( X^G)$.
%Notice that 
%$Y_e \in V_1$, 
%$[Z,Y_e] \in [\mathfrak h,  V_1]$, 
%$\frac{1}{2}[Z,[Z,Y_e]]\in [\mathfrak h,[\mathfrak h,  V_1]]$, and so on.
%We hence use the fact that 
%$[\mathfrak h,\mathfrak h\oplus V_1] \subset \mathfrak h\oplus V_1,$
%to conclude that 
% $$((\bar L_{\tilde g})_* Y^L)_g  
% \in 
% (L_{g'})_*   \pi_\g ( \mathfrak h\oplus V_1)
% = %\subseteq
% (L_{g'})_*   ( V_1)
% =\Delta_{g'}.
% $$
 \qed
 
% Denote by 
% $$\mathcal S := \{\xi\in {\rm Vec}(G) \mid \xi \text{ contact }, \xi(e)=0\}$$ %, i.e., 
%  the space of global contact vector fields on $G$ that vanish at the identity.
  
  For a subspace $W\subseteq \tilde \g$ we   use the notation
  \begin{eqnarray*}
  W ^G:= \{ X^G\in {\rm Vec}(G) \mid X\in W\}. \end{eqnarray*}

 \begin{corollary}
 If $\mathcal S$ denotes the space of global contact vector fields on $G$ that vanish at the identity, we have
 \begin{eqnarray*}
 \mathfrak h ^G \subseteq \mathcal S.
  \end{eqnarray*}
 \end{corollary}
 
 \begin{proof}
 Let $X\in \mathfrak h$. We already proved that $X^G$ is a contact vector field on $G$.
 We only need to verify that $(X^G)_e=0$.
 Since $X^G$ is $i$-related to $\bar X$, it is equivalent to show that $(\bar X )_e=0$, but
 \[(\bar X)_e = \left.  \frac{\dd}{\dd t}\pi(\exp(tX) )\right|_{t=0}
 =
 \left.  \frac{\dd}{\dd t}H\right|_{t=0}
  =0,\]
  as desired.
 \end{proof}

\subsection{A criterion  for Sard's property}

 For $g\in G$, denote $\mathcal S(g)=\{\xi(g)\mid \xi \in \mathcal S\}$. Also, define 
 \begin{eqnarray*}
\mathcal E:= \{g\in G\mid (R_g)_*V_1+(L_g)_* V_1 + \mathcal S(g)=T_g G\}.
 \end{eqnarray*}
Given a horizontal curve $\gamma$ with control $u$, from Section~\ref{EndpointMap} we know that
 \begin{eqnarray*}
(R_{\gamma(1)})_*V_1+(L_{\gamma(1)})_* V_1 + \mathcal S(\gamma(1))\subset {\rm Im}(\dd \End_u)\subset T_{\gamma(1)}G.
 \end{eqnarray*}
Therefore,
if the set $\mathcal E$ is not empty then the abnormal set is a proper subset of $G$. Moreover, observing that $\mathcal E$ is defined by a polynomial relation (see Proposition~\ref{L:bar:contact}), we can deduce that, whenever $\mathcal E$ is not empty then $G$ has  the (Algebraic) Sard Property.

\begin{proposition}\label{criterion_1}
Let $G$ be a Carnot group and let $\tilde G$ and $H$ as in the beginning of Section~\ref{algebraic_prolongation}. Let $\g, \tilde \g$ and  $\mathfrak{h}$ be the corresponding Lie algebras.
Assume that there are $p\in \tilde G$ and $g\in G$ such that
$pH=gH $
%[p]=[g]\in \tilde G/H$
 and
\[  \mathfrak h + V_1 + \Ad_{{p^{-1}}} (    \mathfrak h + V_1 ) = \tilde \g.\]
Then 
\begin{equation}\label{890}
(L_g)_* V_1 + (R_g)_* V_1  +  \mathfrak h ^G(g) = T_g G.
\end{equation}
Moreover, the above formula holds for a nonempty Zariski-open set of points in $G$, and so $G$ has the Algebraic Sard Property.
\end{proposition}
\begin{proof}
Project the equation using $\pi_\g:   \mathfrak h  \oplus  \g  \to \g$ and get
\[   V_1 + \pi_\g\Ad_{{p^{-1}}} (    \mathfrak h + V_1 ) =  \g.\]
Apply the differential of  $\bar L_p \circ \pi_{|_G}$, i.e., the map
 \[  {\rm d} (\bar L_p \circ \pi_{|_G} )_e : \g=T_eG \to T_{[p]}  ( \tilde G / H  ) \]
and get
\[    {\rm d} (\bar L_p \circ \pi_{|_G} )_e V_1 +  {\rm d} (\bar L_p \circ \pi_{|_G} )_e \pi_\g\Ad_{{p^{-1}}} (    \mathfrak h + V_1 ) =  T_{[p]}  ( \tilde G / H  ) .\]
By Equation    \eqref{diff:pi}, the left hand side is equal to
\begin{eqnarray*}
&&\hspace{-4cm}{\rm d} (\bar L_p)_{[e]}  (  {\rm d} i )_e V_1 +  ({\rm d} \pi )_p  (  {\rm d} R_{{p }}) (    \mathfrak h + V_1 ) \\
&=&{\rm d} (\bar L_p)_{[e]}  (  {\rm d} i )_e V_1 +  ({\rm d} \pi )_p  (   (    \mathfrak h + V_1)^R )_p \\
&=&{\rm d} (\bar L_p)_{[e]}  (  {\rm d} i )_e V_1 +  ({\rm d} i )_g  (   (    \mathfrak h + V_1)^G )_g \\
&=&({\rm d} i )_g   {\rm d} ( L_g)_{e}    V_1 +  ({\rm d} i )_g  ( {\rm d} R_{{g }})_e V_1   +   ({\rm d} i )_g       \mathfrak h ^G(g) .    
\end{eqnarray*}
Now \eqref{890} follows because $({\rm d} i )_g $ in an isomorphism. Since \eqref{890} is expressed by polynomial inequations, also the last part of the statement follows. 
\end{proof}

We give an infinitesimal version of the result above.
\begin{proposition}\label{criterion_2}
Assume that 
there exists $\xi \in \tilde \g$ such that
\[  \mathfrak h + V_1 + \ad_{\xi} (    \mathfrak h + V_1 ) = \tilde \g.\]
Then
there are $p\in \tilde G$ and $g\in G$ such that
$pH=gH $
%$[p]=[g]\in \tilde G/H$
 and
\[  \mathfrak h + V_1 + \Ad_{{p^{-1}}} (    \mathfrak h + V_1 ) = \tilde \g.\]
\end{proposition}
\begin{proof}
For all $t>0$, let $p_t:= \exp(t \xi)$.
Take $Y_1,\ldots, Y_m$ a basis of $ \mathfrak h + V_1$.
Let \[Y^t_i:=\Ad_{p_t}(\frac{1}{t} Y_i) =  \ad_\xi (Y_i) + t\sum_{k\geq 1}\dfrac{t^{k-2} (\ad_\xi)^k}{k!}(Y_i).\]
Notice that $Y^t_i\to \ad_\xi (Y_i) $, as $t\to0$.
Then we have
\[ \mathfrak h + V_1+\Ad_{p_t}(  \mathfrak h + V_1) = \span\{Y_1,\ldots, Y_m, Y_1^t,\ldots, Y_m^t   \}.\]
Since 
\[ \span\{Y_1,\ldots, Y_m, Y_1^0,\ldots, Y_m^0   \}= \mathfrak h + V_1 + \ad_{\xi} (    \mathfrak h + V_1 ) = \tilde \g,\]
then $Y_1,\ldots, Y_m, Y_1^t,\ldots, Y_m^t$ span the whole space $\tilde \g$ for $t>0$ small enough.
Moreover, since $p_t\to e\in \tilde G$ and hence $[p_t]\to [e]\in \tilde G/H$, for $t>0$ small enough there exists $g\in G$ such that $[g]=[p_t]$, because $i:G \to \tilde G/H$ is a local diffeomorphism at $e\in G$.
\end{proof}

Combining Proposition~\ref{criterion_1} and \ref{criterion_2} we obtain the following.
\begin{corollary}\label{criterium}
Let $G$ be a Carnot group with Lie algebra $\g$. 
Let $\tilde \g$ and $\mathfrak{h}$ as in  the beginning of Section~\ref{algebraic_prolongation}.
Assume that 
there exists $\xi \in \tilde \g$ such that
\[  \mathfrak h + V_1 + \ad_{\xi} (    \mathfrak h + V_1 ) = \tilde \g.\]
Then $G$ has the Algebraic Sard Property.
%\begin{enumerate}
%\item[i) ] 
% there exists $g\in G$ such that $ (R_g)_*V_1+(L_g)_* V_1 +  \mathfrak h ^G(g)=T_g G$;
%\item[ii) ]  there exists a nonempty Zariski-open set of points $g\in G$ for which 
%$ (R_g)_*V_1+(L_g)_* V_1 +  \mathfrak h ^G(g)=T_g G$;
%\item[iii) ]  there exists a nonempty Zariski-open set of points of $G$ not reached by any abnormal curve leaving from the origin.
%\end{enumerate}
\end{corollary}

\section{Applications} \label{sec:applications}
In this section we use the criteria that we established in Section~\ref{sec:sufficient} in order to prove items (2) to (4) of Theorem~\ref{main_thm1}. The proof of (5) and (6) 
will be based on \eqref{Eq_on_e} and  Corollary~\ref{criterium}.

  The free Lie algebra on $r$ generators is a graded Lie algebra generated freely by an $r$-dimensional vector space $V$. It thus
  has the form 
  \[\mathfrak{f}_{r, \infty} = V \oplus V_2 \oplus V_3 \oplus \ldots\]
  Being free, the general linear group $GL(V)$ acts on this Lie algebra by strata-preserving automorphisms. 
   In order to form the free $k$-step rank $r$ Lie algebra $\mathfrak{f}_{r, k} $ we simply quotient $\mathfrak{f}_{r, \infty}$
  by the Lie ideal $\oplus_{s > k} V_s$.  
  Thus, 
  %as a vector space
%  {\it and a representation space} 
  \[\mathfrak{f}_{r, k} = V \oplus V_2 \oplus \ldots \oplus V_k.\]

\subsection{Proof of (2) and (3)}\label{xi}
We consider  the free nilpotent Lie group $F_{2,4}$ with $2$ generators and step $4$, and the free nilpotent Lie group $F_{3,3}$ with $3$ generators and step $3$. Their Lie algebras are stratified, namely $\mathfrak{f}_{2,4}=V_1\oplus V_2 \oplus V_3 \oplus V_4$ and $\mathfrak{f}_{3,3}=W_1\oplus W_2 \oplus W_3$.
%\subsubsection{$F_{2,4}$ - the free nilpotent group of rank $2$ and step $4$}

The Lie algebra $\mathfrak{f}_{2,4}$ is generated by two vectors, say $X_1,X_2$, in $V_1$, which one can complete to  a basis with
%For the Lie algebra of $F_{2,4}$ we have two generators in $V_1$, say $X_1,X_2$. 
%Using Hall notation, the nontrivial brackets are
%A basis of $\mathfrak{f}_{2,4}$ is given by $X_1$, $X_2$ and 
\begin{eqnarray*}
&X_{21}=[X_2,X_1] &\\
&X_{211}=[X_{21},X_1]& \qquad \quad X_{212}=[X_{21},X_2]\\
&X_{2111}=[X_{211},X_1]& \qquad X_{2112}=[X_{211},X_2] =[X_{212},X_1]\quad \qquad X_{2122}=[X_{212},X_2].
\end{eqnarray*}
We apply Corollary~\ref{criterium} to verify the Algebraic Sard Property for $F_{2,4}$.  We take $\mathfrak{h}$ to be the space of all strata preserving derivations of $\mathfrak{f}_{2,4}$, which in this case are generated by the action of $\mathfrak{gl}(2,\R)$ on $V_1$. 
Choose $\xi = X_2+ X_{212} + X_{2111}$. Then $[\xi,V_1]$ contains the vectors $X_{21}+X_{2112}$ and
$X_{2122}$. Next, consider the basis $\{E_{ij}\mid i,j=1,\dots,2\}$ of $\mathfrak{gl}(2,\R)$, where $E_{ij}$ denotes the matrix that has entry equal to one in the $(i,j)$-position and zero otherwise.
We compute the action of the derivation defined by each one of the $E_{ij}$'s on $\xi$. Abusing of the notation $E_{ij}$ for such derivations, an elementary calculation gives
\begin{eqnarray*}
&E_{11}\xi=X_{212}+3X_{2111}&E_{12}\xi=X_1+ X_{211}\\
&E_{22}\xi= X_2+2X_{212}+X_{2111}
&E_{21}\xi= 2X_{2112}.
\end{eqnarray*}
Since we need to show that $V_1 + \ad_\xi V_1 = \g$, it is enough to prove that $V_2\oplus V_3 \oplus V_4=(\ad_\xi V_1)\,\,{\rm mod} \,V_1$,
which follows from direct verification.

%\subsubsection{$F_{33}$ - the free-nilpotent group of rank $3$ and step $3$}
We consider now the case of the free nilpotent group of rank $3$ and step $3$.
The Lie algebra of $F_{3,3}$ is bracket generated  by three vectors in $W_1$, say $X_1,X_2,X_3$, which give a basis with
%Using Hall notation, the nontrivial bracket relations are
\begin{eqnarray}
\label{Hall_basis}
X_{21}=[X_2,X_1] & X_{31}=[X_3,X_1] & X_{32}=[X_3,X_2]\nonumber\\ 
X_{211}=[X_{21},X_1] & X_{212}=[X_{21},X_2] & X_{213}=[X_{21},X_3]\\\nonumber
X_{311}=[X_{31},X_1]& X_{312}=[X_{31},X_2]& X_{313}=[X_{31},X_3]\\\nonumber
X_{322}=[X_{32},X_2]& X_{323}=[X_{32},X_3].&
\end{eqnarray}
We have the bracket relation $ [X_{32},X_1]=X_{312}-X_{213}$.
We apply Corollary~\ref{criterium} to verify the Algebraic Sard Property for $F_{3,3}$. 
We choose $\xi=X_{21}+X_{31}+X_{32}+X_{312}+X_{213}$, and we consider the action of $\mathfrak h$ on it. In this case $\mathfrak h = \mathfrak{gl}(3,\R)$. Let $E_{ij}\in \mathfrak{gl}(3,\R)$ be the matrix that has entry equal to one in the $(i,j)$-position and zero otherwise. Then the set $\{E_{ij}\mid i,j=1,\dots,3\}$ is a basis of $\mathfrak{gl}(3,\R)$. We compute the action of the elements of this basis on $\xi$. If $i\neq j$ we obtain
\begin{eqnarray*}
&E_{12} \xi =X_{31}+X_{311} \quad E_{13} \xi =-X_{21}+X_{211} \quad E_{23} \xi = X_{21}+2X_{212}\\
&E_{21} \xi =X_{32}+X_{322} \quad E_{31} \xi = -X_{32}-X_{323} \quad E_{32} \xi = X_{31} +2X_{313}
\end{eqnarray*}
whereas if $i=j$
\begin{eqnarray*}
&E_{11}\xi=X_{21}+X_{31}+X_{213}+X_{312} \\
&E_{22}\xi =X_{21}+X_{32}+X_{213}+X_{312} \\
&E_{33}\xi=X_{31}+X_{32}  +X_{213}+X_{312}.
\end{eqnarray*}
Next, we consider $[\xi,V_1]$ and notice that it contains the vectors $v=X_{212}+X_{312}+X_{322}$ and $w=X_{213}+X_{313}+X_{323}$. It is now elementary to verify that the eleven vectors $\{E_{ij}\xi\mid i,j=1,2,3\}$, $v$ and $w$ are linearly independent and therefore are a basis of $W_2\oplus W_3$.
In conclusion, $\xi$ satisfies the hypothesis of Corollary~\ref{criterium}.

\begin{remark}
In the above proof, we had to chose the element $\xi$ properly. This was done considering how $GL(3)$ acts on $F_{3,3}$.
Actually, 
$SL(3)$ acts by graded  automorphisms on $\mathfrak{f}_{3,3}$.  As a consequence
each layer,  $W_1, W_2$ and $W_3$,  form $SL(3)$ representations.  We will see in Section~\ref{DeepOnF33}
that the third  layer $W_3$ 
is isomorphic to ${\mathfrak sl}(3)$ with the adjoint representation of $SL(3)$.
This observation allowed us to find the   element $\xi$.
\end{remark}

\subsection{Semisimple Lie groups and associated polarized groups}\label{iwasawa}

We complete here the proof of Theorem~\ref{main_thm1}. We first recall some standard facts in the theory of semisimple Lie groups. For the details we refer the reader to \cite{Knapp}. 
To be consistent with the standard notation, only in this section we write $G$ for a noncompact semisimple Lie group and $N$ (rather than $G$) for the nilpotent part of a parabolic subgroup.

 If $\theta$ is a Cartan involution of the semisimple Lie algebra $\g$ of $G$, then the Cartan decomposition  is given by the vector space direct sum
\[
\mathfrak{g} = \mathfrak{k}\oplus\mathfrak{p},
\]
where $ \mathfrak{k}$ and $\mathfrak{p}$ are the eigenspaces relative to the two eigenvalues $1$ and $-1$ of $\theta$. We fix a maximal abelian subspace $\mathfrak{a}$ of $\mathfrak{p}$, whose dimension will be denoted by $r$. Let $B$ be the Killing form on $\g$; the bilinear form $\langle X,Y\rangle := -B(X,\theta Y)$ defines a scalar product on $\g$, for which the Cartan decomposition is orthogonal and by which $\mathfrak{a}$ can be identified with its dual $\mathfrak{a}^*$. We fix an order on the system $\Sigma\subset \mathfrak{a}^*$ of nonzero restricted roots of $(\g,\mathfrak{a})$. Let $\mathfrak{m}=\{X\in \mathfrak{k} \mid [X,Y]=0 \,\,\forall Y\in \mathfrak{a}\}$. 
The algebra $\g$ decomposes  as $\g= \mathfrak{m} +\mathfrak{a}+ \oplus_{\alpha\in \Sigma} \mathfrak{g}_\alpha$, where $\g_\alpha$ is the root space relative to $\alpha$.
 We denote by $\Sigma_+$ the subset of positive roots.
 The Lie algebra of $N$, denoted $\mathfrak{n}$, decomposes as the sum of (positive) restricted root spaces $\mathfrak{n}=\oplus_{\alpha\in \Sigma_+} \mathfrak{g}_\alpha$.

%  The space $V= \oplus_{\delta\in \Delta_+} \mathfrak{g}_\delta$, with $\Delta_+$ the set of positive simple roots, provides a stratification on $\mathfrak{n}$, so that $(N,V)$ is a Carnot group.  
%

%We recall that if $\theta$ is a Cartan involution of the semisimple Lie algebra $\g$, then the Cartan decomposition  is given by the vector space direct sum
%$$
%\mathfrak{g} = \mathfrak{k}\oplus\mathfrak{p},
%$$
%where $ \mathfrak{k}$ and $\mathfrak{p}$ are the eigenspaces relative to the two eigenvalues $1$ and $-1$ of $\theta$.
%Let now $\mathfrak{a}$ be a maximal abelian subalgebra of $\mathfrak{p}$, and let $\mathfrak{m}=\{X\in \mathfrak{k} \mid [X,Y]=0 \,\,\forall Y\in \mathfrak{a}\}$.
%The semisimple Lie algebra $\g$  decomposes  as $\g= \mathfrak{a} + \mathfrak{m} + \oplus_{\alpha\in \Sigma} \mathfrak{g}_\alpha$, where $\Sigma$ is the set of nonzero restricted roots, and the $\mathfrak{g}_\alpha$'s, $\alpha\in \Sigma$, denote the corresponding root spaces. Moreover, 

\subsubsection{Proof of (4)}
Denote by $\Pi_+$ the subset of positive simple roots.
 The space $V= \oplus_{\delta\in \Pi_+} \mathfrak{g}_\delta$ provides a stratification of $\mathfrak{n}$,  so that $(N,V)$ is a Carnot group. 
We prove that $(N,V)$ has the Algebraic Sard Property. Let $w$ be a representative in $G$ of the longest element in the analytic Weyl group. 
From \cite[Theorem 6.5]{Knapp} we have 
 $\Ad w^{-1} \bar{\mathfrak{n}}=\mathfrak{n}$,
%\footnote{Thm 6.57 Knapp}, 
where $\bar{\mathfrak{n}}= \oplus_{\alpha\in -\Sigma_+} \mathfrak{g}_\alpha$. 
The Bruhat decomposition of $G$ shows that  $N$ may be identified with the dense open subset $N \bar P$ of the homogeneous space $G/\bar P$, where
$\bar P$ denotes the minimal parabolic subgroup of $G$ containing $\bar N$. Here we wrote $\bar N$ for the connected nilpotent Lie group whose Lie algebra is $\bar{\mathfrak{n}}$.
 Now we apply Proposition~\ref{criterion_1} to $\mathfrak{h}= \mathfrak{m}+  \mathfrak{a} + \bar{\mathfrak{n}}$. From our discussion it follows that $\mathfrak{h} + \Ad w^{-1} \mathfrak{h} = \g$.
 This equality holds true in a small neighborhood of $w$, so by density we can find $p$ in $G$ such that $[p]=[n]$ for some $n\in N$ and for which $\mathfrak{h} + \Ad p^{-1} \mathfrak{h} = \g$. Then by Proposition~\ref{criterion_1} we conclude that 
   the desired Sard's property for $N$ follows.

\subsubsection{Proof of (5)}
From the properties of the Cartan decomposition it follows that $[\mathfrak{p},\mathfrak{p}]=\mathfrak{k}$. Then $(G,\mathfrak{p})$ is a polarized group.
We restrict to the case where $\g$ is the split real form of a complex semisimple Lie algebra.
%\footnote{if $\g$ is not split the method may fail: check possible counter-example}
In order to show that $(G,\mathfrak{p})$ has the Analytic Sard Property, 
 we show that there is $\xi \in \mathfrak{a}$ such that $\ad_{\xi} \mathfrak{p}=\mathfrak{k}$. If this holds, then by a similar argument of that in the proof of Proposition~\ref{criterion_2} we also have $\mathfrak{p} + \Ad_{g}\mathfrak{p}=\g$ for some $g\in G$, from which we deduce the Analytic Sard Property. Let then $\xi$ be a regular element in $\mathfrak{a}$. This implies in particular that $\xi$ is such that $\alpha(\xi)\neq 0$ for every root $\alpha$.
Next, observe that for every $\alpha\in \Sigma$ and  $X\in \g_\alpha$, we may write 
 \begin{eqnarray*}
X= \frac{1}{2}(X-\theta X) +\frac{1}{2}(X+\theta X),
 \end{eqnarray*}
where $X-\theta X\in \mathfrak{p}$ and $X+\theta X\in \mathfrak{k}$.
% Every vector in $\mathfrak{k}$  can be written as $X_\alpha + \theta X_\alpha$, where $X_\alpha\in \mathfrak{g}_\alpha$. Correspondingly, the vector $X_\alpha - \theta X_\alpha$ is in $\mathfrak{p}$ and we have
We obtain
 \begin{eqnarray*}
[\xi,X - \theta X]=\alpha(\xi)X - \theta[\theta \xi, X]= \alpha(\xi)(X+ \theta X).
 \end{eqnarray*}
The assumption that $\g$ is split implies in particular that $\mathfrak{k}$ is generated by vectors of the form $X+\theta X$, with $X$ a nonzero vector in a root space.
Since $\xi$ is regular, it follows that $\ad_{\xi} \mathfrak{p}=\mathfrak{k}$, which concludes the proof. 
\\

We observe that if $\g$ is not split, then we do not find a vector $\xi$ such that $\mathfrak{p}+\ad_{\xi} \mathfrak{p}=\mathfrak{g}$ and so the same proof does not work.
This can be shown, for example, by an explicit calculation on $\g=\mathfrak{su}(1,2)$. 
%Let $G$ be a semisimple Lie group with Lie algebra $\g$.
%Let 
%$\mathfrak{g} = \mathfrak{k}+\mathfrak{p}$ 
%be the Cartan decomposition of $\g$ with respect to an involutive automorphism $\theta$ of $\mathfrak{g}$.
%%whose square is equal to the identity. 
%%So $\theta$ has 
%Namely, the spaces $ \mathfrak{k}$ and $\mathfrak{p}$
%are the  eigenspaces
%corresponding
%to the
%two eigenvalues $1$ and $-1$ of $\theta$.
%Since $ [\mathfrak{p}, \mathfrak{p}] = \mathfrak{k}$, 
%the space $\mathfrak{p}$ defines a  
%bracket generating distribution on 
%$G$ of step $2$.

\subsubsection{Proof of (6)}
We observe that  $(G,\oplus_{\alpha\in \Sigma} \mathfrak{g}_\alpha)$ is a polarized group.
Also in this case we assume that $\g$ is split. This implies that every root space $\g_\alpha$, $\alpha \in \Sigma$, is one dimensional, and that $\mathfrak{m}=\{0\}$. 
We recall that the Killing form  $B$ identifies $\mathfrak{a}$ with $\mathfrak{a}^*$. Let $H_\alpha\in \mathfrak{a}$ be such that $\alpha(H)=B(H_\alpha,H)$ for every $H\in \mathfrak{a}$.
Recall that $[X_\alpha,\theta X_\alpha] = B(X_\alpha,\theta X_\alpha) H_\alpha$ and $B(X_\alpha,\theta X_\alpha)<0$. 
Let $\delta_1,\dots,\delta_r$ be a basis of simple roots, and let $X_{\delta_i}$ be a basis of $\g_{\delta_i}$ for every $i=1,\dots,r$. 
The set of vectors $\{H_{\delta_1},\dots,H_{\delta_r}\}$ is a basis of $\mathfrak{a}$.
Then the vector 
\[
\xi= X_{\delta_1}+\dots +X_{\delta_r}
\]
satisfies $[\xi, \oplus_{\alpha\in \Sigma} \mathfrak{g}_\alpha]\supset \mathfrak{a}$, whence $\oplus_{\alpha\in \Sigma} \mathfrak{g}_\alpha + [\xi, \oplus_{\alpha\in \Sigma} \mathfrak{g}_\alpha] = \mathfrak{g}$. Arguing as in the Proof of (5), we conclude that $(G,\oplus_{\alpha\in \Sigma} \mathfrak{g}_\alpha)$ has the Analytic Sard Property.

%%%%%%%%%%%%%%%%%%%%%%%%%%%%%%%%%%%%%%%
%%%%%%%%%%%%%%%%%%%%%%%%%%%%%%%%%%%%%%%

\subsection{Sard Property for some semidirect products}\label{sec:semi}

In this section we construct polarized groups that are not nilpotent and yet have the Algebraic Sard Property.
These examples are constructed as semidirect products.

%Let $(G,V)$ be a Carnot group with first layer $V$. 

Let 
$\psi :H\to {\rm Aut}(G)$
%{\rm Aut}\mathfrak{g}$ 
be 
an action of a Lie group $H$ on a Lie group $G$, i.e., $\psi$ is a continuous homomorphism from $H$ to 
 the 
group of
%the space of  strata preserving 
automorphisms of 
$G$.
%$\mathfrak{g}$. 
Write $\psi_h$ for $\psi(h)$, for $h\in H$.
The semidirect product $G \rtimes_\psi H$ has product
\begin{equation}\label{prod:semi} (g_1, h_1) \cdot (g_2, h_2)
=(g_1 \psi_{h_1}(g_2) , h_1  h_2). 
\end{equation}
Let $V\subseteq \g$ be a polarization for $G$.
Assume that 
\begin{equation}\label{contact:action}
(\psi_h)_*(V) = V, \quad \text{ for all }h\in H.\end{equation}
%The group $H$
%acts on $G$. 
%Let $\{X_1,\dots,X_s\}$ be a basis of the Lie algebra of $H$. 
We consider the group $G \rtimes_\psi H$ 
%defined as the semidirect product of $H$ acting on $G$, and we 
endowed with the polarization $V\oplus\mathfrak{h}$, where $\mathfrak{h}$ is the Lie algebra of $H$.
\begin{prop}\label{prop:semidirect}
Assume that $G \stackrel{\psi}{\acts} H$ is an action satisfying \eqref{contact:action}.
If $(G,V)$ has the Algebraic Sard Property, so does $(G \rtimes_\psi H,V\oplus\mathfrak{h})$.
\end{prop}
\begin{proof}
We show that ${\rm Abn}_{G \rtimes_\psi H}(e)$ is contained in  ${\rm Abn}_G(e) \cdot H$.
It is a consequence of \eqref{contact:action} that a curve $ \gamma(t)=(g(t),h(t))$  in $\tilde G:=G \rtimes_\psi H$ is horizontal
 with respect to $V+\mathfrak{h}$ if and only if    $g(t)$ is horizontal in $G$ and   $h(t)$ is
 horizontal in $H$.
% Indeed, from \eqref{contact:action} ...

Hence, if $g(1)\notin  {\rm Abn}_G(e)$, i.e., $g$ is not abnormal, 
from \eqref{for:Image:of:dEnd},
we have
\begin{eqnarray*}
(\dd {\rm R}_{  \gamma(1)})_e^{-1} {\rm Im}({\dd}\End_{u_\gamma})&=& {\rm span}\{{\rm Ad}_{  \gamma(t)}(V\oplus\mathfrak{h})\mid t\in [0,1]\}\\
%&=& V+\mathfrak{h} + {\rm span}\{{\rm Ad}_{\tilde \gamma(t)}(V\oplus\mathfrak{h})\mid t\in (0,1]\}\\
&\supseteq& V+\mathfrak{h} + {\rm span}\{{\rm Ad}_{  \gamma(t)}V\mid t\in (0,1]\}\\
&=& V+\mathfrak{h} + {\rm span}\{{\rm Ad}_{(g(t),0)}\Ad_{(0,h(t))}V\mid t\in (0,1]\}\\
&=& V+\mathfrak{h} + {\rm span}\{{\rm Ad}_{(g(t),0)}V\mid t\in (0,1]\} \\
&=& \g+\mathfrak{h},
\end{eqnarray*}
where we used first that
$(g,e_H) \cdot (e_G,h) = (g,h)$ and 
${\rm Ad}_{(e_G,h)}(v,0) = ((\dd \psi_h)_e v,0)$; then we used the assumption \eqref{contact:action}  and the fact
$\Ad_{(g,e_H)}(v,0) = (\Ad_gv, 0)$.
\end{proof}

\begin{remark}\label{rem:semidirect}
If $(G,V)$ is a free nilpotent Lie group for which the Algebraic Sard Property holds, we may take $H$ to be any subgroup of $GL(n,V)$ and apply the proposition above to $G\rtimes H$. 
If $(N,V)$ is a Carnot group as we defined in the first part of Section~\ref{iwasawa}, then $\mathfrak{h}$ may be chosen to be any subalgebra of $\mathfrak{m}\oplus\mathfrak{a}$. In particular, the Algebraic Sard Property holds for exponential growth Lie groups $NA$ if $N$ has step $2$.
\end{remark}

%%%%%%%%%%%%%%%%%%%%%%%%%%%%%%%%%%%%

\section{Step-3 Carnot groups} \label{sec:step:three}

  Our first goal in this section  is to prove 
    Theorem~\ref{thm:step_three} concerning the Sard Property for length minimizers in  Carnot groups of step 3.
  A secondary goal is to motivate 
  the claim made in Example~\ref{rmk33} that the typical abnormal curve  in $F_{3,3}$, the free $3$-step rank-$3$ Carnot group,
  does not lie in any proper subgroup.  
To this purpose we   illustrate the beautiful structure of the abnormal equations in this case.
%  A tertiary goal is to make a few remarks regarding the general free $k$ step Carnot groups which might prove useful in future investigations. 

\subsection{Sard Property for abnormal length minimizers}

In \cite{Tan_Yang_step_3} Tan and Yang 
proved that in sub-Riemannian  step-3 Carnot groups all length minimizing curves are smooth. They also 
claim that in this setting all abnormal length minimizing curves are normal. 
Hence, Theorem~\ref{thm:step_three} would immediately follow from Lemma~\ref{lemma:abn:nor}. 
%We prefer to provide here an independent proof of Theorem~\ref{thm:step_three}
Being  unable to follow some of the proofs
in  \cite{Tan_Yang_step_3},
we prefer to provide here an independent proof of Theorem~\ref{thm:step_three},
 which relies on the weaker claim that every length-minimizing curve is normal in some  Carnot subgroup.
%Nonetheless, we prefer to provide here a proof of Theorem~\ref{thm:step_three} without the use of \cite{Tan_Yang_step_3}.

%Here is a sketch of a proof of the weaker Theorem~\ref{thm:step_three} above. 
\begin{proof}[Proof of  Theorem~\ref{thm:step_three}]

By Lemma~\ref{lemma:abn:nor}, it is enough to estimate the set $\Abn^{lm}_{str}(e)$ of points connected to $e$ by strictly abnormal length minimizers. Let $\gamma$  be such a  curve starting from the origin $e$ of a Carnot group $G$ of step $3$. Since  $\gamma$ is not normal, then it satisfies the Goh condition; in particular, $\gamma$ is contained in the algebraic variety
\[
W^\lambda = \{ g \in G :  \lambda (\Ad_{g}  V_2) = 0  \}
\]
for some $\lambda\in \mathfrak g^*\setminus\{0\}$. We now use Remark \ref{rem:anniGoh}, Remark \ref{rem:W:linear}, and the fact that $G$ is of step-3 to deduce that $\lambda\in V_3^*\setminus\{0\}$ and that $W^\lambda$ is a proper subgroup of $G$. Hence also the accessible set $H^\lambda$ in $W^\lambda$ is a proper Carnot subgroup of $G$.

 Since $\gamma$ is still length minimizing in $H^\lambda$, either $\gamma$ is normal in $H^\lambda$, and we stop, 
 or,  being length minimizing, it is strictly abnormal (i.e., abnormal but not normal) in $H^\lambda$, 
 and we iterate. Eventually, we obtain that $\gamma$ is normal   within a Carnot subgroup. We remark that in this subgroup $\gamma$ may be abnormal or not abnormal.
 We do not need divide the two cases. 
 We decompose
\begin{eqnarray*}
\Abn^{lm}_{str}(e)\subseteq \bigcup_{G'< G}\Abn_{G'}^{nor}(e),
% \cup {\rm SNor}_{G'}(e),
\end{eqnarray*}
where $\Abn_{G'}^{nor}(e)$   is the union of all     curves starting from $e$ that are contained in $G'$, are normal in $G'$, and are abnormal within $G$.

%and
%${\rm SNor}_{G'}(e)$  is the union of strictly normal  (i.e., normal but not abnormal)  curves in $G'$ starting from $e$.

% \begin{eqnarray*}
%\Abn(e)\subseteq \bigcup_{G'< G}\Abn_{G'}^{nor}(e),
%\end{eqnarray*}
%where $\Abn_{G'}^{nor}(e)$   is the union of normal abnormal   curves in $G'$ which are abnormal in $G$.

The idea is now to adapt the argument of Lemma~\ref{lemma:abn:nor} for the union of the sets $\Abn_{G'}^{nor}(e)$.
Carnot subgroups of $G$ are parametrized by the Grassmannian of linear subspaces of $V_1$. The dimension of the subgroup is a semi-algebraic function on the Grassmannian. On each of its level sets $Y_m$, all relevant data (e.g., coefficients of the Hamiltonian equation satisfied by normal length minimizing curves) are real analytic. The dual Lie algebras $\mathfrak{g}'^*$ form an analytic vector bundle over $Y_m$. Denote by $\tau_m$ the total space of this bundle. It is a semi-analytic subset of $T^*_e G$. The time 1 solutions of the Hamiltonian equations with inital data in $\tau_m$ give rise to real analytic maps $ \widetilde{Exp}_m:\tau_m \to L^2([0,1],V)$. Each subgroup has its own
geodesic exponential map, giving rise to an analytic map $Exp_m:\tau_m \to G$. Again,
\begin{eqnarray*}
Exp_m=\End\circ \widetilde{Exp}_m .
\end{eqnarray*}
Every point in $\bigcup_{G'<G}\Abn_{G'}^{nor}(e)$ is a value of some $Exp_m$ where the differential of $\End$ is not onto. Therefore, it is a singular value of $Exp_m$. This constitutes a measure zero sub-analytic subset of $G$. 

%
%For $\bigcup_{G'< G} {\rm SNor}_{G'}(e)$, the situation is similar to that of Proposition~\ref{general2step}, with a similar conclusion: this set is contained in the singular values of an analytic map. Therefore, it is contained in a measure zero sub-analytic subset. Such a subset has codimension at least 1.
\end{proof}

\begin{remark} 
In the free $3$-step Carnot group, we are not able to bound  the codimension of $\Abn^{lm}(e)$  away from $1$. However,  the codimension of $\Abn^{lm}_{str}(e)$ is at least 3.
Actually, in the free $3$-step rank-$r$ group 
$\mathbb F_{r,3}$
this codimension is greater or equal than $r^2- r +1$.
The calculation is similar to the one in Section \ref{sec:contained:codimension}.
Indeed,
by the Witt Formula (see \cite[p.140-142]{Bourbaki_Lie}) the dimension of $\mathbb F_{r,3}$ is
\begin{equation}\label{Witt's formula}
\dim\mathbb F_{r,3} =r+\dfrac{r(r-1)}{2} +\dfrac{r^3-r}{3}.
\end{equation}
In the proof of Theorem~\ref{thm:step_three}, 
we showed that each abnormal geodesic from the origin is in a subgroup, which therefore has codimension bounded by 
$\dim\mathbb F_{r-1,3} $, computable via the Witt Formula \eqref{Witt's formula}.
The collection of all the subgroups of rank $r-1$ can be parametrized via the Grassmannian $Gr(r,r-1)$, which has dimension $r-1$.
Therefore, we compute 
\[
\dim \mathbb F_{r,3}
-
 \dim \mathbb F_{r-1,3}
 -
\dim Gr(r,r-1)
=
r^2 - r +1
.
\]
Notice that $r^2 - r +1 $ equals 3 if $r=2$, and is strictly greater than 7 if $r\geq 3$.
\end{remark}

\subsection{Investigations in the rank-$3$ case}\label{DeepOnF33}

  As said in Section~\ref{sec:applications}, the group $GL(V)$ acts on 
  each strata $V_j$ of the free algebra $\mathfrak{f}_{r, \infty}$.
   So 
each summand $V_j$ breaks up into $GL(V)$ irreducibles.
Also,  the 
$k$-step rank $r$ Lie algebra
 decomposes as a  representation space
  \[\mathfrak{f}_{r, k} = V \oplus V_2 \oplus \ldots \oplus V_k.\]
  The first summand $V$ is the `birthday representation' of $GL(V)$.  
   The second summand is well-known as a $GL(V)$ representation, and in any case is easy to guess:
\[V_2 = \Lambda ^2 V\]
with the Lie bracket $V \times V \to \Lambda^2 V$ being $[v,w] = v \wedge w$.  
The third summand is less well-known and  will be treated momentarily.
First a few more generalities.  Any algebra becomes a Lie algebra when we define the Lie bracket between
two elements to be their commutator.  
So the full tensor algebra $\mathfrak{T} (V) = V \oplus V^{\otimes 2} \oplus V^{\otimes 3} \oplus\dots$ inherits
a Lie algebra structure.  Under this bracket we have $[v,w] = v\otimes w - w \otimes v = v \wedge w$ for
$v, w \in V$.  
 The free Lie algebra over $V$ is the Lie subalgebra  that is Lie-generated by $V$ within the
full tensor algebra $\mathfrak{T} (V)$.   In particular, 
\[V_r \subset V^{\otimes r}.\]
Both the symmetric group $S_r$ on $r$ letters, and the general linear group $GL(V)$ acts on $V^{\otimes r}$.
%According to Weyl, 
By Schur-Weyl duality, see \cite[Exercise 6.30 page 87]{Fulton_Harris}, under the joint action of $GL(V) \times S_r$ the space $V^{\otimes r}$ breaks up completely into irreducibles
and this representation is ``multiplicity free'':  each irreducible occurs at most once.  The irreducibles themselves  are written in  the form $S_{\lambda} (V) \otimes {\rm Specht}(\lambda)$.
Here $\lambda$ is a partition of $r$ and is represented by a Young Tableaux with blank boxes.
Then $S_{\lambda} (V)$ is the irreducible representation of $GL(V)$ corresponding to $\lambda$, whereas ${\rm Specht}(\lambda)$
is the irreducible representation of $S_r$ corresponding to this $\lambda$.  If we are only interested in decomposing $V^{\otimes r}$ into $GL(V)$-irreducibles, what this
means is that each   irreducible $S_{\lambda} (V)$ occurs $\dim({\rm Specht}(\lambda))$ times.  
For example, the representation $S^r (V)$ of symmetric powers of $V$ corresponds to the partition $r = 1 + 1 + 1 + \ldots +1$.
The representation $\Lambda^r (V)$ corresponds to the partition $r = r$.  

To the case at hand, $V_3 \subset V^{\otimes 3}$   corresponds to the partition $3 = 2+1$.  
This representation  is dealt with   in fine detail in \cite[pages 75-76]{Fulton_Harris}.  
We summarize the results  within our context.  The bracket map $V \otimes \Lambda^2 V \to V_3$
which sends $v \otimes \omega \to [v,\omega] = v \otimes \omega - \omega \otimes v$
is onto, but as soon as $\dim(V)>2$ it is not injective due to the Jacobi identity.
We want to describe the image $V_3$ of the bracket map.  There is a  canonical inclusion   $i: V \otimes \Lambda^2 V \to V^{\otimes 3}$, namely the identity
 $v \otimes \omega \mapsto v \otimes \omega$, whose image contains   $V_3$. 
 To cut    $V \otimes \Lambda ^2 V \subset V^{\otimes 3}$ down to  $V_3$ we must add linear conditions which encode the   Jacobi
identity.   Consider the   canonical projection  map $\beta: V^{\otimes 3} \to \Lambda ^3 V$ which sends
$v_1 \otimes v_2 \otimes v_3$ to $v_1 \wedge v_2 \wedge v_3$.  Then the   Jacobi identity is $\beta = 0$, so that $V_3 = im(i) \cap ker(\beta)$.

%$\beta: V \otimes \Lambda^2 V \to \Lambda^3 V$.  We can   define $\beta$  by thinking of
%$\Lambda ^3 V = (\Lambda ^3 V^* ) ^*$.   Then  $\beta(v_1 \otimes (v_2 \wedge v_3) (\omega) = \omega(v_1, v_2, v_3)$,
%for $\omega \in \Lambda ^3 V^*$. 

Let us now go to the specific case of $\dim(V) = 3$.  Here $\dim(V \otimes \Lambda ^2 V) = 3 \times 3 = 9$, whereas
 $\dim(V_3) = 8$.  In this case the Jacobi identity is `one-dimensional'.   We show how to  identify $V_3$ with ${\mathfrak sl}(3)$
by fixing a volume form on $V$.   Write coordinates $x,y,z = x_1, x_2, x_3$ on $V$ and take as the resulting volume form
$\mu =dx_1 \wedge dx_2 \wedge dx_3$.  The choice of form both singles out $SL(3) \subset GL(3) = GL(V)$ and
yields a canonical identification $\Lambda^2 V \cong V^*$ by sending $v \wedge w$ to the one-form $\mu(v, w, \cdot)$.
Thus $V \otimes \Lambda^2 V \cong V \otimes V^* = {\mathfrak {gl}}(V)$  as an $SL(3)$ representation space, with $SL(3) = SL(V)$
acting by conjugation on ${\mathfrak {gl}}(V)$. For example,
$\partial_j \otimes (\partial_1 \wedge \partial_2)$ is sent to the element $\partial_j \otimes dx_3$ under this identification.
One verifies that the kernel of $\beta$ is equal to the span of the identity element $I = \partial_1 \otimes dx_1 + \partial_2 \otimes dx_2 + \partial_3 \otimes dx_3$
under this identification. Thus $V_3 \cong {\mathfrak {gl}}(V)/ \R I$.  Next, observe that as an $SL(V)$
(or $GL(V)$) representation space we have:  $V \otimes V^* = {\mathfrak {sl}}(V) \oplus \R { I}$ where
${\mathfrak {sl}}(V)$ consists of those matrices with trace zero.  Thus $V_3 = {\mathfrak {gl}}(V)/ \R I = {\mathfrak {sl}}(V)$, as   $SL(V)$ representation spaces.
Notice that as $GL(V)$ representation spaces this equality does not hold since the   element $\lambda I \in GL(V)$ acts on $V_3$ by
$\lambda^3 I$, while under conjugation the same element  acts on ${\mathfrak {sl}}(V)$ as the identity. An investigation of what $\ad_\xi$
looks like in relation to this $SL(3)$-equivariant decomposition led to the specific element $\xi$ defined at the end of Section~\ref{xi}.

%We are now ready  to write down  the  
To get to the equations describing  abnormality for $F_{3,3}$,
we write its Lie algebra as
\[\mathfrak{f}_{3,3} = V_1 \oplus V_2 \oplus V_3  =  \R^{3} \oplus \R^{3 *} \oplus {\mathfrak {sl}}(3)\]
and so an element of the dual Lie algebra can be written out as
\[\lambda = (\lambda_1, \lambda_2, \lambda_3) \in \mathfrak{f}_{3,3} ^* = V_1 ^{*} \oplus V_2 ^{*}\oplus V_3 ^{*} =  \R^{3 *} \oplus \R^{3} \oplus {\mathfrak {sl}}(3) ^{*}.\]
For this covector to lie along an   abnormal extremal it must  be $\lambda_1 =0$.

We  partition the  abnormal extremals into two classes: those for which $\lambda_2 \ne 0$, which we
call {\it regular abnormal extremals} following Liu-Sussmann,  and those for which $\lambda_2 = 0$.
The Hamiltonian  
\[H = P_1 P_{23} + P_2 P_{31} + P_3 P_{12}\]
generates all  the regular abnormal extremals.
Here 
\[\lambda_1 = (P_1, P_2, P_3)\]
\[\lambda_2 = (P_{23} ,  P_{31} ,  P_{12}).\]
and
\[P_i = P_{X_i} \quad P_{ij} = P_{X_{ij}} = -P_{ji}\]
where we are following the notation of \eqref{momentum_fns} and \eqref{Hall_basis}.
When we say that $H$ ``generates'' the regular abnormal extremals we mean  two things:
(A) the Hamiltonian flow of $H$ preserves the locus $\lambda_1 = 0$, i.e.,   the locus $\Delta ^{\perp} = \{ P_1 = P_2 = P_3 = 0 \}$
and (B) on the locus $\lambda_1 = 0$,  $\lambda_2 \ne 0$, a  unique - up to reparameterization - abnormal extremal passes through every point,
with the extremal  through $(0, \lambda_2, \lambda_3)$ being  the  solution to Hamilton's
equations for this Hamiltionian $H$  with initial conditions $\lambda$.
%(We thank  Igor Zelenko for showing us this Hamiltonian trick.) 
%to  characterize  the regular abnormal extremals.

We follow a Hamiltonian trick that Igor Zelenko kindly showed us for both finding  $H$
and for validating  claims (A) and (B). Start with the   Maximum Principle characterization of  
   abnormal extremals   discussed in Section~\ref{Ham_formalism}.  According to this principle,
an abnormal with control $u(t)$ is a solution to Hamilton's equations
having the time dependent  Hamiltonian $H_u = u_1 P_1 + u_2 P_2 + u_3 P_3$
and lying in the common level set $P_1 = 0, P_2 = 0, P_3 =0$.  From Hamilton's equations
we find
that 
\[\dot P_1 = \{P_1, H_u\} = - u_2 P_{12} - u_3 P_{13}\]
\[\dot P_2 = \{P_2, H_u\} = -u_1 P_{21} - u_3 P_{23}\]
\[\dot P_3 = \{P_3, H_u \} = -u_1 P_{31} -u_2 P_{32}\]
But we must have that $\dot P_i = 0$.  Consequently
$(u_1, u_2, u_3)$ must lie in the kernel of the skew-symmetric matrix
whose entries are $P_{ij}$.  As long as this matrix is not identically zero,
its kernel is one-dimensional and is spanned by $(P_{23}, P_{31}, P_{12})$. It follows that:
\[(u_1, u_2, u_3) = f (P_{23}, P_{31}, P_{12}), f \ne 0.\]
Since the parameterization of the abnormal is immaterial, we may  take  $f = 1$.
Plugging our expression for $u$ back in to $H_u$ yields the form of $H$ above.

We  can  write down the ODEs governing the regular abnormal extremals, using this $H$.
We have just seen that
\[u = \lambda_2 = (P_{23}, P_{31}, P_{12})\]
describes  the controls, i.e., the moving element of $V$.   
This control evolves according to
\[
\dot u = A u
\]
where $A$ is a constant matrix in $SL(3)$.
These are to be supplemented by the understanding of what
the resulting abnormal extremal is 
\[\lambda_1 = 0, \lambda_2 = u, \lambda_3 = A.\]
We want to establish  Hamilton's equations, using this $H$.  
%To establish this law,
 For doing so, we compute 
% Hamilton's equations 
 $\dot P_{ij} = \{P_{ij}, H \}$
 and $\dot P_{ijk} = \{P_{ijk}, H\} = 0$ where $P_{ijk} = P_{X_{ijk}}$.  The first equation
 results in a bilinear pairing between $P_{ij}$ and $P_{ijk}$
 which, when the $P_{ijk}$ are properly interpreted as an element
 $A \in SL(3)$, is matrix multiplication.
 
\subsection{Computation of  abnormals not lying in any subgroup}
\label{not lying in any subgroup}
 Take a diagonalizable $A$ with distinct nonzero eigenvalues $a, b, c$, 
   $a+ b+c  =0$.  For simplicity, let it be 
   ${\rm diag}(a,b,c)$ relative to our choice of coordinates for $V$.
     Then $u$ evolves according to 
     $u(t) = (A e^{at}, B e^{bt}, C e^{ct})$.
     We may suppose that none of $A, B, C$ are zero by assuming that
     no components of $\lambda_2 = u(0)$ are zero. 
     The corresponding curve in $G$ passing through $e = 0$, 
     projected onto the first level  is 
  the curve $x_1 = \frac{1}{a} (A(e^{at} -1)$, $x_2 = \frac{1}{b} (B(e^{bt} -1)$, $x_3 = \frac{1}{c} (C(e^{ct} -1)$.
  Since the       functions $1, e^{at}, e^{bt}, e^{ct}$ are linearly independent ,
  the curve projected to  the first level  
   cannot lie in any proper subspace of $V$, which in turn implies  that the entire abnormal  curve cannot lie in any proper subgroup of
   $G$.  

Alternatively, one can directly use
Corollary~\ref{right_invariant_interpretation}.
In fact,  with the notation of 
Section~\ref{sec:applications}, 
 one can take $\lambda =e_{21}^* - e_{31}^* + e_{32}^* - c e_{213}^* + b e^*_{312}$ to prove that the curve with control
 $u(t) = ( e^{(-b -c)t},  e^{bt},  e^{ct})$ is abnormal.

 \subsubsection{The characteristic viewpoint}

We put forth one further  perspective on abnormal extremals which makes the computation  just done
more transparent.  Take any polarized manifold $(Q,\Delta)$.  Take the annihilator bundle of $\Delta$, denoted $\Delta^{\perp} \subset T^*Q$.
Restrict the canonical symplectic form $\omega$ of $T^*Q$ to $\Delta^{\perp}$.  Call this
restriction $\omega_{\Delta}$.  Then the abnormal extremals are precisely the (absolutely continuous)
characteristics for $\omega_{\Delta}$, that is the curves in $\Delta^{\perp}$ whose tangents are a.e. in $\ker(\omega_{\Delta})$.   
Let $\pi: \Delta^{\perp} \to Q$ be the canonical projection.  Then a linear algebra computation shows that $d \pi_{(q, \lambda)}$ projects
$\ker(\omega_{\Delta})(q, \lambda)$  linearly isomorphically onto $\ker(w_q (\lambda)) \subset \Delta_q$
where $\lambda \in \Delta_q ^{\perp} \mapsto w_q (\lambda) \in \Lambda ^2 \Delta_q ^*$ is the operator called the
``dual curvature'' in \cite{Montgomery}. In the case of a polarized group $(Q, \Delta) = (G, V)$
we have that $w_q (\lambda)$ is the   two-form of Equation \eqref{curv_map} for $\lambda = \eta \in V^{\perp}$. 

In our situation $V$ has dimension $3$ so that   $w(\lambda)$ has either rank 2 or 0 and thus  its
kernel has dimension  1 or 3.    The kernel has dimension 1 exactly when $\lambda_2 \ne 0$, and rank 3 exactly when $\lambda_2 = 0$.
Along the points where $\lambda_2 \ne 0$ the kernel of $\omega_{\Delta}$ is a line field, and the Hamiltonian vector field $X_H$
for $H$ above rectifies this line field.  Note that $X_H$ vanishes exactly along the variety $\lambda_2 = 0$.

\section{Open problems}  \label{sec:open}					 
Is $\Abn(e)$ , the set of endpoints of abnormal extremals
leaving the identity, a closed analytic variety in $G$ when $G$ is a   simply connected   polarized Lie group?  
 In all   examples  computed, the answer is `yes'.  
 However,  even the following   more basic  questions are still open.
 
Is $\Abn(e)$ closed?
		
Can $\Abn(e)$ be the entire group $G$?

Concerning the importance of the adjective  ``simply connected'' above, consider the torus. 
Any integrable distribution $V$ whose corank is 1 or greater    on any space  $G$ has its $\Abn(e)$ the leaf through $e$.
Consequently an irrationally oriented polarization $V$ on the torus has for its $\Abn(e)$ a set that  is neither closed nor analytic.

We also wonder wether statements 5 and 6 of Theorem~\ref{main_thm1} can be upgraded to algebraic.		

Can one unify (6) and (7) having the result for all semisimple groups?

If $G$ and $H$ are polarized Lie groups having the Sard Property, does any semidirect product $G\rtimes H$ have  the Sard Property?
		
Finally, in the particular case of rank 2 Carnot groups, what is the minimal codimension of $\Abn(e)$?

%%%%%%%%%%%%%%%%%%%%%%%%%%%%%%%%%%%%

%\section{References that may be added} 
%
%
%
% A. Agrachev \& J.-P. Gauthier, On subanalyticity of Carnot-Carath«eodory distances, Ann. Inst. H.
%Poincar«e Anal. Non Lin«eaire

%%%%%%%%%%%%%%%%%%%%%%%%%%%%%%%%%%
%%%%%%%%%%%%%%%%%%%%%%%%%%%%%%%%%

% \begin{thebibliography}{10}
% 
% 
%\bibitem[Buc08]{Bucicovschi}
%Orest Bucicovschi
%\emph{Simple {L}ie algebras, algebraic prolongations and contact structures},
%Thesis (Ph.D.)ÐUniversity of California, San Diego. 2008. 96 pp. 
%
%\bibitem[OW10]{OW}
%Alessandro Ottazzi and Ben Warhurst, \emph{Contact and $1$-quasiconformal maps on Carnot groups}, J. Lie Theory.   \textbf{21} (2011), 787--811.
%
%
% 
% \bibitem[Tan70]{Tanaka}
%Noboru Tanaka, \emph{On differential systems, graded {L}ie algebras and
%  pseudogroups}, J. Math. Kyoto Univ. \textbf{10} (1970), 1--82.
%  
%  \bibitem[Yam93]{Yamaguchi}
%Keizo Yamaguchi, \emph{Differential systems associated with simple graded {L}ie
%  algebras}, Progress in differential geometry, Adv. Stud. Pure Math., vol.~22,
%  Math. Soc. Japan, Tokyo, 1993, pp.~413--494.
%
%\end{thebibliography}

\bibliography{general_bibliography}
%\bibliography{LMOPV_biblio_DONOTMODIFY}
%\bibliographystyle{abbrv}
\bibliographystyle{amsalpha}
\end{document}